\documentclass[a4paper, 11pt, oneside, reqno]{amsart}
\usepackage{ifpdf}
\ifpdf 
    \usepackage[pdftex]{graphicx}   
    \usepackage[pdftex,     
            plainpages=false,   
            breaklinks=true,    
            colorlinks=true,
            pdftitle=My Document
            pdfauthor=My Good Self
            colorlinks=true, 
        urlcolor=blue, 
        citecolor=red,
        linkcolor=blue
           ]{hyperref} 
\else 
    \usepackage{graphicx}       
    \usepackage{hyperref}       
\fi 

\usepackage{lmodern}

\hypersetup{
    linkcolor=blue
}
\usepackage{inputenc}
\usepackage{amsfonts,amsmath, stmaryrd} 
\usepackage{amssymb}
\usepackage{verbatim}
\usepackage{amsopn}
\usepackage[english]{babel}
\usepackage{amsthm}
\usepackage{esint}
\usepackage{setspace}
\usepackage{todonotes}
\usepackage{nicefrac}
\usepackage{enumerate}
\usepackage{mathrsfs} 
\usepackage{mathtools}
\usepackage{color}
\usepackage{bbm}
\usepackage{soul}

\usepackage{calc}
\usepackage{lmodern}

\usepackage{graphicx, subfig, wrapfig}
\usepackage{epigraph}

\date{\today}

\theoremstyle{definition} \newtheorem{definition}{Definition}[section]
\theoremstyle{definition} \newtheorem{remark}[definition]{Remark}
\theoremstyle{plain} \newtheorem{lemma}[definition]{Lemma}
\theoremstyle{plain} \newtheorem{proposition}[definition]{Proposition}
\theoremstyle{plain} \newtheorem{theorem}[definition]{Theorem}
\theoremstyle{plain} \newtheorem*{theorem*}{Theorem}
\theoremstyle{plain} 
\theoremstyle{plain} \newtheorem*{mthm*}{Main Theorem}
\theoremstyle{plain} \newtheorem*{conjecture*}{Conjecture}
\theoremstyle{plain} \newtheorem{conjecture}[definition]{Conjecture}
\theoremstyle{plain} 
\theoremstyle{plain} \newtheorem*{problem*}{Problem}
\theoremstyle{plain} \newtheorem{corollary}[definition]{Corollary}
\theoremstyle{definition} 
\theoremstyle{plain} 

\newcommand{\R}{\mathbb{R}}

\newcommand{\1}{\mathbbm{1}}
\renewcommand{\ss}[1]{\mathscr{#1}}

\newcommand{\rr}[1]{\mathrm{#1}}
\newcommand{\rest}{\mathbin{\vrule height 1.6ex depth 0pt width
0.13ex\vrule height 0.13ex depth 0pt width 1.3ex}}

\renewcommand{\L}{\mathscr{L}}
\renewcommand{\H}{\mathscr{H}}
\newcommand{\Q}{\mathbb{Q}}
\newcommand{\N}{\mathbb{N}}
\newcommand{\Z}{\mathbb{Z}}
\newcommand{\B}{\mathscr B}

\DeclareMathOperator{\BV}{BV}

\newcommand{\eps}{\varepsilon}
\newcommand{\e}{\varepsilon}

\definecolor{orange}{rgb}{1,0.5,0}

\renewcommand{\tilde}{\widetilde}
\newcommand{\Conn}{{\rm Conn}}
\newcommand{\Lip}{{\rm Lip}}

\newcommand{\dive}{\mathop{\mathrm{div}}}

\numberwithin{equation}{section} 

\title[Renormalization for nearly incompressible BV vector fields]{Renormalization for autonomous nearly incompressible BV vector fields in 2D} 

\author{Stefano Bianchini}
\address{S. Bianchini: S.I.S.S.A., via Bonomea 265, 34136 Trieste, Italy}
\email{bianchin@sissa.it}

\author{Paolo Bonicatto}
\address{P. Bonicatto: S.I.S.S.A., via Bonomea 265, 34136 Trieste, Italy}
\email{paolo.bonicatto@sissa.it}

\author{Nikolay A. Gusev}
\address{N. A. Gusev: Dybenko st., 22/3, 94, 125475 Moscow, Russia}
\email{n.a.gusev@gmail.com}

\begin{document}

\begin{abstract}
Given a bounded autonomous vector field $b \colon \R^2 \to \R^2$, we study the uniqueness of bounded solutions to the initial value problem for the related transport equation
\begin{equation*}
\partial_t u + b \cdot \nabla u= 0.
\end{equation*}
Assuming that $b$ is of class BV and it is nearly incompressible, we prove uniqueness of weak solutions to the transport equation. The starting point is the result which has been obtained in \cite{BG} (where the \emph{steady nearly incompressible} case is treated). Our proof is based on splitting the equation onto a suitable partition of the plane: this technique was introduced in \cite{ABC1}, using the results on the structure of level sets of Lipschitz maps obtained in \cite{ABC2}. Furthermore, in order to construct the partition, we use Ambrosio's superposition principle \cite{AmbrInv}. \\

\noindent \textsc{Keywords}: transport equation, continuity equation, renormalization, disintegration of measures, Lipschitz functions, Superposition Principle. \\

\noindent \textsc{MSC (2010): 35F10, 35L03, 28A50, 35D30.}
\end{abstract}

\thanks{S. Bianchini and P. Bonicatto have been partially supported by the PRIN project ``Nonlinear Hyperbolic Partial Differential Equations, Dispersive and Transport Equations: theoretical and applicative aspects''. N.A. Gusev was partially supported by the Russian Foundation for Basic Research, project no. 13-01-12460. The authors wish to thank one of the referees for an useful observation which leads to Remark \ref{rem:referee}}

\maketitle

\begin{center}
Preprint SISSA 67/2014/MATE
\end{center}

\section{Introduction and notation}
In this paper we consider the \emph{continuity equation}
\begin{equation}\label{eq:continuity-intro}
\partial_{t} u + \dive (ub) = 0
\end{equation}
and the \emph{transport equation}
\begin{equation}\label{eq:transport-intro}
\partial_t u + b \cdot \nabla u = 0,
\end{equation}
for a scalar field $u\colon I \times \R^{2} \to \R$ (where $I=(0,T)$, $T>0$) with a vector field $b \colon I\times\R^{2} \to \R^{2}$.
We study the initial value problems for these equations with the same initial condition
\begin{equation}\label{eq:initial-condition}
u(0,\cdot) = \overline{u}(\cdot),
\end{equation}
where $\bar u \colon \R^2 \to \R$ is a given scalar field.

Our aim is to investigate uniqueness of weak solutions to \eqref{eq:continuity-intro}, \eqref{eq:initial-condition} (and to \eqref{eq:transport-intro}, \eqref{eq:initial-condition}) under weak regularity assumptions on the vector field $b$.

Even if we are interested only to the two dimensional case, we present here the main definitions in $\R^d$, with $d \in \N$. When $b\in L^\infty(I\times \R^d)$ then \eqref{eq:continuity-intro} is understood in the standard sense of distributions:
$u\in L^\infty(I\times \R^d)$ is called a \emph{weak solution} of the continuity equation if  \eqref{eq:continuity-intro} holds in $\ss D'(I\times \R^d)$.
One can prove (see e.g. \cite{camillonote}) that, if $u$ is a weak solution of \eqref{eq:continuity-intro}, then there exists a map $\widetilde{u} \in L^{\infty}([0,T] \times \R^{d})$ such that $u(t, \cdot) = \widetilde{u}(t, \cdot)$ for a.e. $t \in I$ and $t \mapsto \widetilde{u}(t, \cdot)$ is weakly$^{\star}$ continuous from $[0,T]$ into $L^{\infty}(\R^{d})$. This allows us to prescribe an initial condition \eqref{eq:initial-condition} for a weak solution $u$ of the continuity equation in the following sense: we say that $u(0,\cdot)=\bar u(\cdot)$ holds if $\widetilde{u}(0, \cdot) = \bar u(\cdot)$.

Definition of weak solutions of the transport equation \eqref{eq:transport-intro} is slightly more delicate. If the divergence of $b$ is absolutely continuous with
respect to the Lebesgue measure then \eqref{eq:transport-intro} can be written as
\[
\partial_t u + \dive(ub) - u \dive b=0,
\]
and the latter equation
can be understood in the sense of distributions (see e.g. \cite{dipernalions} for the details). 
We are interested in the case when $\dive b$ is not absolutely continuous. In this case the notion of weak solution of \eqref{eq:transport-intro}
can be defined for the class of \emph{nearly incompressible vector fields}.

\begin{definition} \label{def:ni} A bounded, locally integrable vector field $b \colon I \times \R^d \to \R^{d}$ is called \emph{nearly incompressible} if there exists a function $\rho \colon I \times \Omega \to \R$ (called \emph{density} of $b$) and a constant $C>0$ such that $C^{-1} \le \rho(t,x) \le C$ for $\L^1 \times \L^d$-a.e. $(t,x) \in I \times \Omega$ and 
\begin{equation}\label{eq:rho_cont}
\partial_{t} \rho + \dive(\rho b) = 0 \qquad \text{ in }  \mathscr D^{\prime}(I \times \Omega).
\end{equation}
\end{definition}

Nearly incompressible vector fields were introduced in connection with the hyperbolic conservation laws,
namely, the Keyfitz-Kranzer system \cite{keyfitz}. See e.g. \cite{camillonote} for the details.
Using mollification one can prove that if $\dive b \in L^\infty(I \times \R^d)$ then $b$ is nearly incompressible.
The converse implication does not hold, so near incompressibility can be considered as a weaker version of the assumption $\dive b \in L^\infty(I \times \R^d)$.

\begin{definition}
\label{def:def-sol-near-inc}
Let $b$ be a nearly incompressible vector field with density $\rho$.
We say that a function $u \in L^{\infty}(I \times \R^{2})$ is a \emph{($\rho$--)weak solution} of \eqref{eq:transport-intro} if
\begin{equation*} 
(\rho u)_{t} + \dive (\rho u b) = 0 \quad \text{ in } \ss D^{\prime}(I \times \R^{2}).
\end{equation*}
\end{definition}

Thanks to Definition~\ref{def:def-sol-near-inc} one can prescribe the initial condition for a $\rho$--weak solution of the transport equation
similarly to the case of the continuity equation, which we mentioned above (see \cite{camillonote} for the details).

\emph{Existence} of weak solutions to initial value problem for transport equation with a nearly incompressible vector field can be proved by a standard regularization argument \cite{camillonote}. The problem of \emph{uniqueness} of weak solutions is much more delicate.
The theory of uniqueness in the non-smooth framework has started with the seminal paper of R.J. DiPerna and P.-L. Lions \cite{dipernalions} where
uniqueness was obtained as a corollary of
so-called \emph{renormalization property} for the vector fields with Sobolev regularity.
Thanks to Definition~\ref{def:def-sol-near-inc} the renormalization property can be defined also for nearly incompressible vector fields:

\begin{definition}\label{def:ren-n-i} We say that a nearly incompressible vector field $b$ with density $\rho$ has the \emph{renormalization property} if for every $\rho$--weak solution $u \in L^{\infty}(I \times \R^{d})$ of \eqref{eq:transport-intro}
and any function $\beta \in C^1(\R)$ the function $\beta(u)$ also is a $\rho$-weak solution of \eqref{eq:transport-intro}, i.e. it satisfies
\begin{equation*}
\partial_{t}\left(\rho\beta(u) \right) + \dive \left(\rho \beta(u) b\right) = 0 \quad \text{ in } \ss D^{\prime}(I \times \R^{d}).
\end{equation*}
\end{definition}

Nearly incompressible vector fields are related to a conjecture, made by A. Bressan in \cite{bressan1}:
\begin{conjecture}[Bressan's compactness conjecture]
Let $b_n\colon \mathbb R \times \R^d \to \mathbb R^d$, $n\in \mathbb N$, be a sequence of smooth vector fields.
Denote by $\Phi_n$ the solutions of the ODEs
$$
\begin{aligned}
\frac{d}{dt} \Phi_n(t,x) &= b_n(t, \Phi_n(t,x)), \\
\Phi_n(0,x) &= x.
\end{aligned}
$$
Assume that $\|b_n\|_\infty + \|\nabla_{t,x} b_n\|_{L^1}$ is uniformly bounded and there exists a constant $C>0$
such that
$$
C^{-1} \le \det(\nabla_x \Phi_n(t,x)) \le C
$$
for all $(t,x) \in \mathbb R \times \mathbb R^d$ and all $n \in \mathbb N$.
Then the sequence $\Phi_n$ is strongly precompact in $L^1_{\mathrm{loc}}$.
\end{conjecture}
It has been proved in \cite{ambrosiobouchutdelellis} that Bressan's conjecture would follow from the next one: 
\begin{conjecture}[Renormalization conjecture]\label{Bressan}
Any bounded, nearly incompressible vector field ${b \in \BV_{\rm loc}(\R \times \R^d)}$
has the renormalization property (in the sense of Definition \ref{def:ren-n-i}). 
\end{conjecture}

The renormalization property can also be generalized for the systems of transport equations.
Moreover, if $\eta$ is another density of the nearly incompressible vector field $b$ and
$b$ has the renormalization property with the density $\rho$, then any $\rho$--weak solution of \eqref{eq:transport-intro}
is also an $\eta$--weak solution and vice versa. In other words, the property of being a $\rho$--weak solution does not depend on the
choice of the density $\rho$ provided that renormalization holds. We refer to \cite{camillonote} for the details.

If the functions $\rho$, $u$ and $b$ were smooth, renormalization property would be an easy corollary of the chain rule.
Out of the smooth setting, the validity of this property is a key step to get uniqueness of weak solutions.
Indeed, if we for simplicity consider $\mathbb T^d$ instead of $\mathbb R^d$, then integrating the equation above over the torus we get
\begin{equation*}
\partial_t \int_{\mathbb T^d} \rho \beta(u) \, dx = 0.
\end{equation*}
So if $\bar u = 0$ then for $\beta(y)=y^2$ we get
\begin{equation*}
\int_{\mathbb T^d} \rho(t,x) u^2(t,x) \, dx = 0
\end{equation*}
for a.e. $t$ which implies $u(t,\cdot)=0$ for a.e. $t$.

The problem of uniqueness of solutions is thus shifted to prove the renormalization property for $b$: in \cite{dipernalions} the authors proved that renormalization property holds under Sobolev regularity assumptions; some years later, L. Ambrosio \cite{AmbrInv} improved this result, showing that renormalization holds for vector fields which are of class $\BV$ (locally in space) and have absolutely continuous divergence. \\
Another approach giving explicit compactness estimates has been introduced in \cite{MR2369485}, and further developed in \cite{bouchut_crippa,Jabin}: see also the references therein.

In the two dimensional autonomous case the problem of uniqueness is addressed in the papers \cite{ABC1}, \cite{ABC2} and \cite{BG}. Indeed, in two dimensions and for divergence-free autonomous vector fields, renormalization theorems are available even under mild assumptions, because of the underlying Hamiltonian structure. In \cite{ABC1}, the authors characterize the autonomous, divergence-free vector fields $b$ on the plane such that the Cauchy problem for the continuity equation \eqref{eq:continuity-intro} admits a unique bounded weak solution for every bounded initial datum \eqref{eq:initial-condition}. The characterization they present relies on the so called \emph{Weak Sard Property}, which is a (weaker) measure theoretic version of Sard's Lemma. Since the problem admits a Hamiltonian potential, uniqueness is proved following a strategy based on splitting the equation on the level sets of this function, reducing thus to a one-dimensional problem. This approach requires a preliminary study on the structure of level sets of Lipschitz maps defined on $\R^{2}$, which is carried out in the paper \cite{ABC2}. \\

In \cite{BG} the \emph{steady nearly incompressible} autonomous vector fields on $\Omega = \R^2$ were considered. Namely, an autonomous vector field $b\colon \mathbb R^2 \to \mathbb R^2$ is called steady nearly incompressible if it admits a steady density $\tilde\rho$, i.e. there exists a function $\tilde \rho$, uniformly bounded from below and above by some strictly positive constants, such that $\dive(\tilde \rho b) = 0$. 
It was proved in \cite{BG} that any steady nearly incompressible BV vector field on $\mathbb R^2$ has the renormalization property.
In the present paper we extend this result to the non-steady case.
Any steady nearly incompressible vector field is nearly incompressible, but the inverse implication does not hold in general.
For instance, consider a vector field $b \colon (0, 2) \to \mathbb R$ given by $b(x) = |x-1| - 1$. If it was steady nearly incompressible,
the function $\tilde \rho \cdot b$ would be constant on $(0,2)$ and thus $\tilde \rho$ could not be uniformly bounded from above by a positive constant.
On the other hand this vector field $b$ is nearly incompressible: the solution to the continuity equation $\partial_t \rho + \partial_x (\rho b) =0$
with the initial condition $\rho|_{t=0}=1$ satisfies $e^{-t} \le \rho(t,x) \le e^t$, as one can easily demostrate using the classical method of characteristics, since $b$ is Lipschitz. This simple example can be generalized to higher dimensions.




The main result of this paper is a partial answer to the Conjecture~\ref{Bressan}:

\begin{mthm*} Every bounded, autonomous, compactly supported, nearly incompressible $\BV$ vector field on $\R^2$ has the renormalization property.\end{mthm*}

In particular, we obtain the following 
\begin{corollary} Suppose that $b\colon \R^2 \to \mathbb R^{2}$ is a compactly supported, nearly incompressible $\BV$ vector field (with density $\rho$).
Then 
\begin{enumerate}
\item $\forall u_0 \in L^{\infty}(\R^2)$ there exists a unique ($\rho$-)weak solution $u\in L^{\infty}(I \times \R^2)$ to the transport equation \eqref{eq:transport-intro} with the initial condition $u|_{t=0} = u_0$.
\item $\forall u_0 \in L^{\infty}(\R^2)$ there exists a unique weak solution $u\in L^{\infty}(I \times \R^2)$ to the continuity equation \eqref{eq:continuity-intro} with the initial condition $u|_{t=0} = u_0$.
\end{enumerate}
\end{corollary}

\subsection{Structure of the paper}

The paper is organised as follows. 

In Section \ref{s:new-one} we present Ambrosio's Superposition Principle. By this Principle, the measure $\rho(t, \cdot)\L^2$ (where $\rho$ is a nonnegative bounded solution of the continuity equation \eqref{eq:rho_cont}) can be represented as an 
\emph{image} of some probability measure $\eta$ on the space of curves $C([0,T];\R^2)$ (concentrated on the solutions of the ODE $\gamma^\prime = b(\gamma)$) under the evaluation map $e_t \colon \gamma \mapsto \gamma(t)$: 
\begin{equation*}
\rho(t,\cdot) \L^2 = (e_{t})_{\#} \eta.
\end{equation*}

Using this Theorem, we construct a suitable partition of the plane and we reduce our problem \emph{locally} to the case when the density $\rho$ is steady, which has been studied in \cite{BG}. In this case, since $\dive (\rho b)=0$, there exists
a Lipschitz \emph{Hamiltonian} $H\colon \R^2 \to \R$ such that
\begin{equation*}
\rho b = \nabla^\perp H,
\end{equation*}
where $\nabla^\perp = (-\partial_2, \partial_1)$.

In the general nearly incompressible case it is not possible to construct the
Hamiltonian $H$ directly as in the case of steady density. However, we reduce the problem to the steady case using the following argument. Suppose that a nonnegative bounded function $\varrho$ solves the continuity equation
\begin{equation*}
\varrho_t + \dive(\varrho b) = 0,
\end{equation*}
$t\mapsto \varrho(t,\cdot)$ is weak* continuous and
for some open set $\Omega$ and $t_{1,2}\in [0,T]$ we have
$\varrho(t_1, \cdot) = \varrho(t_2,\cdot) = 0$ a.e. on $\Omega$.
Integrating the continuity equation with respect to time on $[t_1,t_2]$
it is easy to see that
\begin{equation*}
r(x):=\int_{t_1}^{t_2} \varrho(t,x) \, dt
\end{equation*}
solves
\[
\dive (r b) = 0
\]
in $\ss{D}'(\Omega)$. Therefore in $\Omega$ one can construct a \emph{local Hamiltonian}
$H_\Omega$ such that
\begin{equation*}
r b = \nabla^\perp H_\Omega
\end{equation*}
in $\Omega$.

Once we have constructed the local Hamiltonians, we show how we can split an equation of the form 
\begin{equation}\label{divub=mu}
\dive (ub) = \mu, \qquad u\colon \R^2 \to \R
\end{equation}
where $\mu$ is a measure on $\R^2$, into an equivalent family of equations along the level sets of $H$. This is done in Section \ref{S_recent_resu}, where we also recall the main results of \cite{ABC2,ABC1,BG} and adapt them to our setting. In Section~\ref{s:new-four} we establish the so-called Weak Sard Property for the Hamiltonian $H$. 

Then we turn to study in detail the relationship between level sets of the local Hamiltonian $H$ and the trajectories of the vector field $b$: in Section \ref{s:new-five}, we present some lemmas which show that (up to a $\eta$ negligible set) all non constant integral curves of $b$ are contained in ``good'' level sets of $H$. 

In Section \ref{s:new-six} we prove that the divergence operator is \emph{local}, in the sense that the measure $\mu$ in \eqref{divub=mu} vanishes on the set $M:=\{b= 0\}$ (Proposition \ref{prop:locality}). We stress that this result is true for every space dimension and it is crucial to obtain a better description of the link between the level sets and the trajectories. This is achieved in Section \ref{s:new-seven}, where in particular, we prove that ``good'' level sets of $H$ cover almost all the set $M^{c}=\{b \ne 0\}$. 

Finally, in Section \ref{s:new-eight} we first show how the time-dependent problem
\begin{equation}\label{eq:time-dependent-problem-intro}
\begin{cases}
u_{t} + b \cdot \nabla u = 0, \\
u(0, \cdot) = u_{0}(\cdot),
\end{cases} \qquad \text{in } \ss D^{\prime}((0,T) \times \R^2).
\end{equation}
can be reduced to a family of one-dimensional problems on level sets of the Hamiltonians, which can be solved explicitly. This allows to construct a $\eta$-negligible set $R$ of trajectories with the following property, which is reminiscent of the standard Method of Characteristics (within the smooth setting): if $u$ is a solution of \ref{eq:time-dependent-problem-intro}, then for all $\gamma \notin R$ the function $t \mapsto u(t,\gamma(t))$ is constant. This crucial result (Lemma \ref{l-characteristics-1}) combined with an elementary observation (Lemma \ref{l-characteristics-2}) concludes in Section \ref{s:new-nine} the proof of the {\bf Main Theorem} (Theorem \ref{T_main}).

\subsection{Notation}\label{ss-notation}
Throughout the paper, we use the following notation:
\begin{itemize}
\item $(X,d)$ is a metric space; 
\item $\1_{E}$ is the characteristic function of the set $E \subset X$, defined as $\1_{E}(x)=1$ if $x \in E$ and $\1_{E}(x)=0$ otherwise; 
\item $\Omega$ denotes in general a simply connected open set in $\R^{2}$;
\item ${\rm dist}(x, E)$ is the distance of $x$ from the set $E$, defined as the infimum of $d(x,y)$ as $y$ varies in $E$; 
\item $B(x,r)$ or, equivalently, $B_{r}(x)$ is the open ball in $\R^{d}$ with radius $r$ and centre $x$; 
$B(r)$ is the open ball in $\R^{d}$ with radius $r$ and centre $0$;
\item $\fint_{E} f \, d\mu$ denotes the \emph{average} of the function $f$ over the set $E$ with respect to the positive measure $\mu$, that is 
\begin{equation*}
\fint_{E} f \, d\mu := \frac{1}{\mu(E)}\int_{E} f \, d\mu,
\end{equation*}
\item $\mu \rest A$ denotes the restriction of a measure $\mu$ on a set $A$.
\item $\vert \mu \vert$ is the total variation of a measure $\mu$;
\item $\mu^\mathrm{sing}$ the singular component of $\mu$ with respect to the Lebesgue measure;
\item $\L^{d}$ is the Lebesgue measure on $\R^{d}$ and $\H^{k}$ is the $k$-dimensional Hausdorff measure;
\item $\text{Lip}(X)$ is the space of real-valued Lipschitz functions; 
$\text{Lip}_c(X)$ is the space of real-valued compactly supported Lipschitz functions;
\item $C_{c}^{\infty}(\Omega)$ is the space of smooth compactly supported functions, also called \emph{test functions};
\item $\BV(\Omega)$ set of functions with bounded variation;
\item $\mathscr D^{\prime}(\Omega)$ is the space of distributions on the open set $\Omega$;


\item $\Gamma:= C([0,T]; \R^2)$ will denote the set of continuous curves in $\R^2$; 
\item $\dot{\Gamma}:= \left\{\gamma \in \Gamma: \gamma(t) = \gamma(0), \, \forall t \in [0,T]\right\}$ denotes the set of constant curves (whose graphs are fixed points); 
\item $\tilde{\Gamma}:= \Gamma\setminus \dot{\Gamma}$ denotes the set of non-constant curves (whose graphs have positive length);
\item $e_{t} \colon \Gamma \to \R^2$ is the \emph{evaluation map} at time $t$, i.e. $e_t(\gamma)= \gamma(t)$.
\end{itemize}
Moreover, if $A \subset \R^2$ is a measurable set, 
\begin{itemize}
\item $\Gamma_{A} := \left\{\gamma \in \Gamma: \L^{1}(\{ t \in [0,T]: \gamma(t) \in A\})>0\right\}$ denotes the set of curves which stay in $A$ for a positive amount of time; 
\item $\tilde{\Gamma}_{A} := \tilde{\Gamma} \cap \Gamma_{A}$ denotes the set of non-constant curves which stay in $A$ for a positive amount of time; 
\item $\dot{\Gamma}_{A} := \dot{\Gamma} \cap \Gamma_{A}$ denotes the set of constant curves which stay in $A$ for a positive amount of time.
\item for every $s \in [0,T]$, we denote by 
\begin{align*}
& \Gamma^{s}_{A} := \left\{ \gamma \in \Gamma: \gamma(s) \in A \right\}, \\
& \tilde{\Gamma}^{s}_{A} :=\left\{ \gamma \in \tilde{\Gamma}: \gamma(s) \in A \right\}, \\
&\dot{\Gamma}^{s}_{A} :=\left\{ \gamma \in \dot{\Gamma}: \, \gamma(s) \in A \right\}
\end{align*}
accordingly the sets of all curves, non-constant curves and constant curves, which at time $s$ belong to $A$;
\item 
$\mathsf T_{A} := \left\{ \gamma \in \Gamma_A: \; \gamma(0) \notin A, \gamma(T) \notin A\right\}$ denotes
the set of curves which stay in $A$ for a positive amount of time and have the endpoints outside $A$.
\end{itemize}
If $A \subseteq \R^2$, we denote by 
\begin{equation*}
\begin{split}
& \Conn(A) := \Big\{ C \subset A: C \text{ is a connected component of } A\Big\}, \\ 
& \Conn^{\star}(A):=\Big\{C \in \Conn(A): \H^1(C)>0\Big\},
\end{split}
\end{equation*}
and 
\begin{equation*}
A^{\star} := \bigcup_{C \in \Conn^{\star}(A)} C.
\end{equation*}

When the measure is not specified, it is assumed to be the Lebesgue measure, and we often write
\begin{equation*}
 \int f(x)\,dx
\end{equation*}
for the integral of $f$ with respect to $\L^{d}$.

\subsection{Disintegration of a measure} 

Let $\mu$ be a Radon measure on a metric space $X$. Let $Y$ be a metric space and let $f\colon X\to Y$ be a Borel function.
We denote by $f_\# \mu$ the \emph{image measure} of $\mu$ under the map $f$. In particular, for any $\varphi\in C_c(Y)$
we have
\[
\int_X \varphi(f(x)) \, d\mu(x) = \int_Y \varphi(y) \, d (f_\# \mu)(y).
\]

Let $\nu$ be a Radon measure on $Y$ such that $f_\# |\mu| \ll \nu$.
According to the Disintegration Theorem (Theorem 2.28 of \cite{AFP} or for the most general statement Section 452 of \cite{Fre:measuretheory4}) there exists a unique measurable family of Radon
measures $\{\mu_y\}_{y\in Y}$ such that for $\nu$-a.e. $y\in Y$ the measure $\mu_y$ is concentrated on the level set $f^{-1}(y)$ and
\begin{equation*}
\mu = \int_{y} \mu_y \, d \nu (y),
\end{equation*}
that is, for any $\varphi \in C_c(X)$
\begin{equation*}
\int_X \varphi(x)\, d\mu(x) = \int_Y \left( \int_X \varphi(x) \, d\mu_y(x) \right) \, d \nu (y).
\end{equation*}
The family $\{\mu_y\}_{y\in Y}$ is called the \emph{disintegration of $\mu$ with respect to $f$ (and $\nu$)}.

\subsection{Coarea formula}

Suppose that $H \colon \R^2 \to \R$ is a Lipschitz function. Coarea formula (see e.g. \cite{AFP} for the general statement) provides further information about the structure of the disintegration of $|\nabla H| \L^2$ with respect to $H$:
\begin{lemma}\label{lemma-coarea}
Let $\{\varpi_h\}_{h\in \R}$ denote the disintegration of the measure $|\nabla H| \L^2$ with respect to $H$ and let $E_h:=H^{-1}(h)$. Then for a.e. $h\in \R$ we have $\H^1(E_h)<\infty$ and $\varpi_h = \H^1 \rest E_h$. In other words, the disintegration of $|\nabla H| \L^2$ with respect to $H$ is given by
\begin{equation*}
|\nabla H| \L^2 = \int_\R \H^1 \rest E_h \, dh.
\end{equation*}
\end{lemma}

\section{Setting of the problem}\label{s:new-one}

\subsection{Ambrosio's Superposition Principle} In \cite{AmbrInv}, L. Ambrosio proved the \emph{Superposition Principle}. Since we will use it later on in this section, we present here the statement. Let us consider the continuity equation in the form 
\begin{equation}\label{eq:continuity-superposition}
\begin{cases}
\partial_{t} \mu_{t} + \dive(b \mu_{t}) = 0, \\
\mu_{0} = \overline{\mu},
\end{cases}
\end{equation}
where $[0,T] \ni t \mapsto \mu_{t}$ is a measure valued function and $b \colon [0,T] \times \R^{d} \to \R^{d}$ is a bounded, Borel vector field. A solution to \eqref{eq:continuity-superposition} has to be understood in distributional sense. 

We have the following 

\begin{theorem}[Superposition Principle]\label{thm:superposition} Let $b \colon [0,T] \times \R^{d} \to \R^{d}$ be a bounded, Borel vector field and let $[0,T] \ni t \mapsto \mu_{t}$ be a positive, locally finite, measure-valued solution of the continuity equation \eqref{eq:continuity-superposition}. Then there exists a family of probability measures $\{\eta_{x}\}_{x \in \R^{d}}$ on $\Gamma$ such that 
\begin{equation*}
\mu_{t} = \int {e_{t}}_{\#} \eta_{x} d\bar{\mu}(x),
\end{equation*}
for any $t \in (0,T)$ and $\left( {e_{0}}\right)_{\#}\eta_{x}=\delta_{x}$. Moreover, $\eta_{x}$ is concentrated on absolutely continuous integral solutions of the ODE starting from $x$, for $\overline{\mu}$-a.e. $x \in \R^{d}$.
\end{theorem}

In other words, any nonnegative measure-valued solution $\mu_t$ of the continuity equation \eqref{eq:continuity-superposition}
can be represented as
\begin{equation}\label{e-eta-def}
\mu_t = {e_t}_\# \eta,
\end{equation}
where $\eta$ is some nonnegative measure on the space of continuous curves $\Gamma$, which is concentrated on the integral curves
of the vector field $b$. In terms of Theorem~\ref{thm:superposition} this measure $\eta$ can be defined by
\begin{equation*}
\eta = \int_{\R^d} \eta_x \, d\bar\mu(x).
\end{equation*}
(I.e. the family $\{\eta_x\}_{x\in \R^d}$ is the disintegration of $\eta$ under the map $e_0$.)

\subsection{Partition and curves}
\label{Ss_parti_curves}

Let $b \colon \R^2 \to \R^{2}$ be an \emph{autonomous}, nearly incompressible vector field, with $b \in \BV(\R^2) \cap L^{\infty}(\R^2)$; we assume $b$ is compactly supported (with support in the unit ball of $\R^2$, $\mathbb B:=B(0,1)$), defined everywhere and Borel. Let us consider the countable covering $\mathscr B$ of $\R^{2}$ given by 
\begin{equation*}
\mathscr B:= \left\{ B(x,r): \, x \in \Q^{2}, r \in \Q^{+} \right\}. 
\end{equation*}
For each ball $B \in \mathscr B$, we are interested to the trajectories of $b$ which cross $B$, staying inside $B$ for a positive amount of time. We therefore define, for every ball $B \in \mathscr B$ and for every rational numbers $s,t \in \Q \cap (0,T)$ such that $s<t$, the sets
\begin{equation*}
\mathsf T_{B,s,t} := \left\{ \gamma \in \Gamma_{B}: \, \gamma(s) \notin B, \gamma(t) \notin B\right\}.
\end{equation*}
We recall that (see Notations) 
\begin{equation*}
\Gamma_{B} := \left\{\gamma \in \Gamma: \L^{1}(\{ t \in [0,T]: \gamma(t) \in B\})>0\right\}. 
\end{equation*}

In this first section we will work for simplicity with the sets $T_B:=T_{B,0,T}$, where $B \in \mathscr B$ (and without any loss of generality we assume $T \in \Q$). 

\begin{remark}
\label{rem:unione-tb-tutto-moving} It is fairly easy to see that
\begin{equation*}
\bigcup_{B \in \mathscr B} \mathsf T_{B} = \tilde{\Gamma}.
\end{equation*}
Indeed, for every curve which is moving there exists a point $\gamma(t) \not= \gamma(0), \gamma(T)$, so that one has just to choose a ball in $\mathscr B$ containing $\gamma(t)$ but not $\gamma(0)$, $\gamma(T)$.
\end{remark}


By Definition \ref{def:ni}, there exists a function $\rho \colon [0,T]\times \R^2 \to \R$ which satisfies continuity equation \eqref{eq:rho_cont} in $\mathscr D^{\prime}((0,T) \times \R^2)$.
Therefore, by Ambrosio's Superposition Principle \ref{thm:superposition}, there exists a measure $\eta$ on $\Gamma$, concentrated on the set of trajectories of $b$, such that 
\begin{equation}\label{eq:rho-e-eta}
\rho(t,\cdot) \L^{2} = ({e_{t}})_{\#} \eta ,
\end{equation}
where we recall that $e_{t} \colon \Gamma \to \R^2$ is the evaluation map $\gamma \mapsto \gamma(t)$. For a fixed ball $B \in \mathscr B$, we consider the measure $\eta_{B}:=\eta  \rest \mathsf T_B$ and we define $\rho_{B}$ by $\rho_{B}(t,\cdot)\L^{2} = {(e_{t})}_{\#} \eta_{B}$. Then we set
\begin{equation}
\label{E_def_r_B}
r_{B}(x) := \int_{0}^{T} \rho_{B}(t,x)dt, \qquad x \in B.
\end{equation}

\begin{lemma} \label{lemma:divrb} It holds $\dive(r_{B}b)=0$ in $\mathscr D^{\prime}(B)$.\end{lemma}

\begin{proof} For any $\phi \in C_{c}^{\infty}(B)$ we have
\begin{equation*}
\begin{split}
\int_{B} r_{B} b(x) \cdot \nabla \phi(x) dx & = \int_{B}\int_{0}^{T} \rho_{B}(t,x)b(x) \cdot \nabla \phi(x) \, dt \, dx \\ 
                             & = \int_{0}^{T}\int_{\mathsf T_{B}} b(\gamma(t)) \cdot (\nabla \phi)(\gamma(t)) \, d\eta_{B} \,dt \\
                             & = \int_{0}^{T}\int_{\mathsf T_{B}} \dot{\gamma}(t)  \cdot (\nabla \phi)(\gamma(t)) \, d\eta_{B} \,dt \\
                             & = \int_{0}^{T}\int_{\mathsf T_{B}} \frac{d}{dt}\phi(\gamma(t)) \, d\eta_{B} \,dt \\ 
                             & = \int_{\mathsf T_{B}} \Big[\phi(\gamma(T))-\phi(\gamma(0))\Big] \, d\eta_{B} = 0.
\end{split}
\end{equation*}
because for $\eta_{B}$-a.e. $ \gamma \in \mathsf T_{B}$, $\gamma(0) \notin B$, $\gamma(T) \notin B$.
\end{proof}

\section{Recent results for uniqueness in the two dimensional case}
\label{S_recent_resu}
We recall here some facts about uniqueness of bounded solutions for the continuity equation in the two dimensional case, following in particular \cite{ABC2, ABC1}.  

\subsection{Structure of level sets of Lipschitz functions}

Let $\Omega \subset \R^{2}$ be a bounded, open set and let $f \colon \Omega \to \R$ be a Lipschitz function. For any $r \in \mathbb R$, we denote by $E_r:=f^{-1}(r)$ the corresponding level set. 

\begin{theorem}[{\cite[Thm. 2.5]{ABC2}}] \label{thm:ABC2_structure} Suppose that $f \colon \Omega \to \R$ is a compactly supported Lipschitz function. For any $r \in \mathbb R$, let $E_r:=f^{-1}(r)$. Then the following statements hold for $\mathscr L^1$-a.e. $r \in f(\Omega)$: 
\begin{enumerate}
\item[\rm (1)] $\H^1(E_r)<\infty$ and $E_r$ is countable $\H^{1}$-rectifiable; 
\item[\rm (2)] for $\H^{1}$-a.e. $x \in E_{r}$ the function $f$ is differentiable at $x$ with $\nabla f(x) \ne 0$; 
\item[\rm (3)] $\Conn^{\star}(E_r)$ is countable and every $C \in \Conn^{*}(E_r)$ is a closed simple curve;
\item[\rm (4)] $\H^1(E_r \setminus E^{*}_r)=0$.
\end{enumerate}
\end{theorem}

For brevity, we will say that the level set $E_r$ is \emph{regular with respect to} $\Omega$ if it satisfies conditions (1)-(2)-(3)-(4) (or it is empty). In this way, the theorem above can be stated by saying that for a.e. $r \in \R$ the level sets $E_r$ are regular with respect to $\Omega$.

\subsection{Disintegration of Lebesgue measure with respect to Hamiltonians}
\label{ss:disintegration-lebesgue}

From Lemma \ref{lemma:divrb} we have $\dive(rb)=0$ in $B$; since $B$ is simply connected, there exists a Lipschitz potential $H_{B} \colon B \to \R$ such that
\begin{equation*}
\nabla^{\perp}H_{B}(x) = r_{B}(x)b(x), \qquad \text{for } \L^{2}\text{-a.e. } x \in B.
\end{equation*}
Using Theorem \ref{thm:ABC2_structure} on the Lipschitz function $H_B$, we can define the negligible set $N_1$ such that $E_h$ is regular in $B$ whenever $h\notin N_1$; moreover, let $N_{2}$ denote the negligible set on which the measure $((H_B)_\# \L^2)^\rr{sing}$ is concentrated, where $((H_B)_\# \L^2)^\rr{sing}$ is the singular part of $((H_B)_\# \L^2)$ with respect to $\L^{1}$. Then we set 
\begin{equation}\label{Ndef}
N:=N_1 \cup N_2
\qquad
\text{and}
\qquad 
E^*:= \cup_{h\notin N} E_h^*
\end{equation}
Therefore we can associate to $B$ a triple $(H_B,N,E)$. 
For any $x\in E$ let $C_x$ denote the connected component of $E$ such that $x\in C_x$.
By definition of $E$ for any $x \in E$ the corresponding connected component $C_x$ has strictly positive length.

Let us fix an arbitrary ball $B\in \mathscr B$. For brevity let $H$ denote the corresponding Hamiltonian $H_B$.

\begin{lemma}[{\cite[Lemma 2.8]{ABC1}}] \label{disint-in-ball} 
There exist Borel families of measures $\sigma_h,\kappa_h$, $h \in \R$, such that
\begin{equation} \label{disint-lebesgue}
\L^2 \rest B = \int \left( c_h \H^1 \rest E_h + \sigma_h \right) \, d h
+ \int \kappa_h \, d\zeta(h),
\end{equation}
where
\begin{enumerate}
\item $c_h \in L^1(\H^1 \rest E^\star_h)$, $c_h >0$ a.e.; moreover, by Coarea formula, we have $c_h=1/|\nabla H|$
a.e. (w.r.t. $\H^1\rest E^\star_h$);
\item $\sigma_h$ is concentrated on $E^\star_h \cap \{\nabla H = 0\}$;
\item $\kappa_h$ is concentrated on $E^\star_h \cap \{\nabla H = 0\}$; 
\item \label{Point_4_lemma_dis_bal} $\zeta := H_\# \L^2 \rest (B\setminus E^{\star})$ is concentrated on $N$ (hence $\zeta  \perp \L^1$);
\item $\sigma_{h}$ is concentrated on $E_{h} \cap \{b \ne 0, r_{B} = 0\}$.
\end{enumerate}
\end{lemma}

\begin{proof}
Points (1)-(4) are exactly \cite[Lemma 2.8]{ABC1}. Concerning (5), it can be proved using minor modifications of the proof of \cite[Theorem 8.2]{BG}: indeed, we have that, being $b$ of class BV and hence approximately differentiable a.e., $H_\# \L^2 \rest \{b=0\} \perp \L^1$: by comparing two disintegrations of $\L^2\rest \{b=0\}$ we conclude that $\sigma_h$ is concentrated on $\{b \ne 0\}$ for a.e. $h$.\end{proof}

\begin{remark}
\label{rem:h1-perp-sigma}
Using Coarea formula (see Lemma \ref{lemma-coarea}), we can show
\begin{equation*}
\H^1(E_h \cap \{\nabla H = 0\}) = 0
\end{equation*}
for $\L^1$-a.e. $h\notin N$. Therefore $\sigma_h \perp \H^1$ for $\L^1$-a.e. $h\notin N$.
\end{remark}

\begin{remark}\label{orientation-by-b}
Thanks to \eqref{disint-lebesgue} we always can add to $N$, if necessary, an $\L^1$-negligible set so that for any $h\notin N$ for $\H^1$-a.e. $x\in E_h^\star$
we have $r(x) > 0$, $b(x)\ne 0$ and  $r(x)b(x) = \nabla^\perp H(x)$.
\end{remark}
 
\subsection{Reduction of the equation on the level sets.}

Our goal is now to study the equation $\dive (ub)= \mu$, where $u$ is a bounded Borel function on $\R^2$ and $\mu$ is a Radon measure on $\R^2$, inside a ball from the collection $\mathscr B$. 


\begin{lemma}\label{disint-div-ub}
Suppose that $\mu$ is a Radon measure on $\R^2$ and $u \in L^{\infty}(\R^2)$. Then equation
\begin{equation}\label{eq:div-ub-first-time}
\dive (ub) = \mu 
\end{equation}
holds in $\ss D^{\prime}(B)$ if and only if:
\begin{itemize}
\item the disintegration of $\mu$ with respect to $H$ has the form
\begin{equation} \label{H-disint-mu}
\mu = \int \mu_h \, dh + \int \nu_{h} \, d \zeta(h),
\end{equation}
where $\zeta$ is defined in Point \eqref{Point_4_lemma_dis_bal} of Lemma~\ref{disint-in-ball};
\item for $\L^1$-a.e. $h$
\begin{equation}\label{div-on-level-set}
\dive \left( u c_h b \H^1 \rest E_h \right) + \dive (u b \sigma_h) = \mu_h;
\end{equation}
\item for $\zeta$-a.e. $h$
\begin{equation}\label{eq:div-on--level-set-bad-term}
\dive (u b \kappa_h) = \nu_h.
\end{equation}
\end{itemize}
\end{lemma}

\begin{proof} Let $\lambda^{s}$ be a measure on $\R$ such that $H_{\#} \vert \mu \vert \ll \L^{1} + \zeta + \lambda^{s}$, where $\zeta$ is defined as in Lemma~\ref{disint-in-ball} and $\lambda^{s} \perp \L^{1} + \zeta$. Applying the Disintegration Theorem, we have that 
\begin{equation}\label{eq:disint-mu-general}
\mu = \int \mu_{h}dh + \int \nu_{h} d\zeta(h) + \int \lambda_{h} d\lambda^{s}(h),
\end{equation}
with $\mu_h, \nu_h, \lambda_h$ concentrated on $\{H = h\}$.
Writing equation \eqref{eq:div-ub-first-time} in distribution form we get 
\begin{equation*}
 \int_{\R^2} u (b \cdot \nabla \phi) \, dx + \int \phi d\mu = 0, \qquad \forall \phi \in C_{c}^{\infty}(B).
\end{equation*}
By an elementary approximation argument, it is clear that we can use as test functions $\phi$ Lipschitz with compact support.

Using the disintegration of Lebesgue measure \eqref{disint-lebesgue} and the disintegration \eqref{eq:disint-mu-general} we thus obtain
\begin{equation}
\label{E_weak_disint_lev}
\begin{split}
\int \left[ \int_{\R^2} u c_{h} (b \cdot \nabla \phi) \, d\H^{1} \rest E_{h} + \int_{\R^2} u (b \cdot \nabla \phi) \, d\sigma_{h} \right] dh & \\
+ \int \int_{\R^2} u (b \cdot \nabla \phi) \, d \kappa_{h} \, d\zeta(h) + \int \int_{\R^2} \phi \, d\mu_{h}\, dh & \\
+ \int \int_{\R^2} \phi \, d\nu_{h} \, d\zeta(h) + \int \int_{\R^2} \phi \, d\lambda_{h} \, d\lambda^{s}(h) &= 0,
\end{split}
\end{equation}
for every $\phi \in \mathrm{Lip}_{c}(B)$. In particular we can take
\begin{equation*}
\phi = \psi(H(x)) \varphi(x), \qquad \psi \in C^\infty(\R), \ \varphi \in C^\infty_c(B),
\end{equation*}
so that we can rewrite \eqref{E_weak_disint_lev} as
\begin{equation*}
\begin{split}
\int \psi(h) \left[ \int_{\R^2} u c_{h} (b \cdot \nabla \varphi) \, d\H^{1} \rest E_{h} + \int_{\R^2} u (b \cdot \nabla \varphi) \, d\sigma_{h} \right) dh & \\
+ \int \psi(h) \int_{\R^2} u (b \cdot \nabla \varphi) \, d \kappa_{h} \, d\zeta(h) + \int \psi(h) \int_{\R^2} \varphi \, d\mu_{h}\, dh & \\
+ \int \psi(h) \int_{\R^2} \varphi \, d\nu_{h} \, d\zeta(h) + \int \psi(h) \int_{\R^2} \varphi \, d\lambda_{h} \, d\lambda^{s}(h) &= 0,
\end{split}
\end{equation*}
because
\begin{equation*}
b \cdot \nabla \phi = \psi(H(x)) b \cdot \nabla \varphi(x)
\end{equation*}
for $\L^2$-a.e. $x \in \R^2$.

Since the equalities above hold for all $\psi \in C^\infty (\R)$ we have
\begin{equation*}
\int \left[ \int_{\R^2} u c_{h} (b \cdot \nabla \varphi) \, d\H^{1} \rest E_{h} + \int_{\R^2} u (b \cdot \nabla \varphi)\, d\sigma_{h} \right] dh + \int \int_{\R^2} \varphi \, d\mu_{h}\, dh = 0,
\end{equation*}
\begin{equation*}
\int \left[ \int_{\R^2} u (b \cdot \nabla \varphi) \, d \kappa_{h}+\int_{\R^2} \varphi \, d\nu_{h} \right] \, d\zeta(h) = 0, 
\end{equation*}
\begin{equation*} 
\int \int_{\R^2} \varphi \, d\lambda_{h} \, d\lambda^{s}(h) = 0, 
\end{equation*}
which give, respectively, \eqref{div-on-level-set}, \eqref{eq:div-on--level-set-bad-term} and \eqref{H-disint-mu}.

\end{proof}

\subsection{Reduction on connected components of level sets.}
If $K\subset \R^d$ is a compact then, in general, not any connected component $C$ of $K$ can be separated from $K\setminus C$ by a smooth function.
However, it can be separated by a sequence of such functions:
\begin{lemma}[{\cite[Section 2.8]{ABC2}}, {\cite[Lemma 5.3]{BG}}]\label{lemma:sep-by-test} If $K \subset \R^{d}$ is compact then for any connected component $C$ of $K$ there exists a sequence $(\phi_{n})_{n \in \N} \subset C_{c}^{\infty}(\R^{d})$ such that 
\begin{enumerate}
\item $0 \le \phi_{n} \le 1$ on $\R^{d}$ and $\phi_{n} \in \{0,1\}$ on $K$ for all $n \in \N$; 
\item for any $x \in C$, we have $\phi_{n}(x)=1$ for every $n \in \mathbb N$;
\item for any $x \in K \setminus C$, we have $\phi_{n}(x) \to 0$ as $n \to +\infty$; 
\item for any $n \in \N$, we have ${\rm supp}\, \nabla \phi_{n} \cap K = \emptyset$.  
\end{enumerate}
\end{lemma}

With the aid of this lemma we can now study the equation \eqref{div-on-level-set} on the nontrivial connected components of the level sets. 
In view of Lemma~\ref{disint-div-ub} in what follows we always assume that $h\notin N$ (see \eqref{Ndef}).

\begin{lemma}\label{cc-lemma}
The equation \eqref{div-on-level-set} holds iff
\begin{itemize}
\item for any nontrivial connected component $C$ of $E_h$ it holds 
\begin{equation}\label{div-on-cc}
\dive \left(u c_h b \H^1 \rest C \right) + \dive (u b \sigma_h \rest C) = \mu_h\rest C;
\end{equation}
\item it holds 
\begin{equation}\label{div-on-cc-parte-brutta}
\dive(u b \sigma_h \rest (E_h \setminus E_h^{\star})) = \mu_h \rest (E_h \setminus E_h^{\star}).
\end{equation}
\end{itemize}
\end{lemma}

\begin{proof} For any Borel set $A \subset \R^2$ we introduce the following functional
\begin{equation*}
\Lambda_{A}(\psi) := \int_{A} u c_{h} (b \cdot \nabla \psi) \, d\H^{1}\rest E_{h} + \int_{A} u (b \cdot \nabla \psi) \, d\sigma_{h} + \int_{A} \psi \, d\mu_{h},
\end{equation*}
for all $\psi \in C_{c}^{\infty}(B)$.

Now fix a connected component $C$ of $E_{h}$ and take a sequence of functions $(\phi_{n})_{n \in \mathbb N}$ given by Lemma \ref{lemma:sep-by-test} (applied with $K:=E_{h}$). By assumption, we have that $\Lambda(\psi \phi_{n})=0$ for every $\psi \in C_{c}^{\infty}(B)$ and for every $n$. 
Let us pass to the limit as $n\to \infty$.

On one hand we have 
\begin{equation*} 
\int \psi \phi_{n} \,d\mu_{h} = \int_{C} \psi \, d\mu + \int_{E_{h} \setminus C} \psi \phi_{n} \, d\mu \to \int_{C} \psi \, d\mu
\end{equation*}
because the second term converges to $0$ since $\phi_n\to 0$ pointwise on $E_h\setminus C$.

On the other hand $\nabla(\psi \phi_{n}) = \psi \nabla \phi_{n} + \phi_{n} \nabla\psi$. 
In the terms with $\phi_{n} \nabla \psi$ we pass to the limit as above. The terms with the product $\psi \nabla \phi_{n}$ identically vanish thanks to the condition (4) on $\phi_{n}$ in Lemma \ref{lemma:sep-by-test}. Therefore, we have that for every $\psi \in C_{c}^{\infty}(B)$
\begin{equation*}
\Lambda_{E_{h}}(\psi \phi_{n}) \to \int_{C} u c_{h} (b \cdot \nabla \psi) \, d\H^{1} + \int_{C} u (b \cdot \nabla \psi) \, d\sigma_{h} + \int_{C} \psi \, d\mu_{h} = \Lambda_{C}(\psi),
\end{equation*}
as $n \to +\infty$. Since $\Lambda_{E_{h}}(\psi \phi_{n}) = 0$ for every $n$, we deduce that $\Lambda_{C}(\psi)=0$ and this gives \eqref{div-on-cc}.

In order to get \eqref{div-on-cc-parte-brutta}, it is enough to observe that $E_{h}^{\star}$ is a countable union of connected components $C$, therefore (from the previous step) we deduce that
\begin{equation*}
\int_{E^{\star}_{h}} u c_{h} (b \cdot \nabla \psi) \, d\H^{1} + \int_{E^{\star}_{h}} u (b \cdot \nabla \psi)\, d\sigma_{h} + \int_{E^{\star}_{h}} \psi \, d\mu_{h} = 0, \qquad \forall \psi \in C_{c}^{\infty}(B).
\end{equation*}
Hence 
\begin{equation*}
\Lambda_{ E_{h} \setminus E^{\star}_{h}} := \int_{E^{\star}_{h} \setminus E_{h}} u c_{h} (b \cdot \nabla \psi) \, d\H^{1} + \int_{E^{\star}_{h}\setminus E_{h}} u (b \cdot \nabla \psi) \, d\sigma_{h} + \int_{E^{\star}_{h}\setminus E_{h}} \psi \, d\mu_{h} = 0, 
\end{equation*}
for every $\psi \in C_{c}^{\infty}(B)$. Remembering that $\H^{1}(E^{\star}_{h} \setminus E_{h}) = 0$ by Theorem \ref{thm:ABC2_structure} we get \eqref{div-on-cc-parte-brutta} and this concludes the proof.

The converse implication can be easily obtained by summing the equations \eqref{div-on-cc} and \eqref{div-on-cc-parte-brutta}.\end{proof}

\begin{lemma}\label{separation}
Equation \eqref{div-on-cc} holds iff
\begin{subequations}
\begin{equation}\label{div-on-level-set-reg-term}
\dive \left( u c_h b \H^1 \rest C \right) = \mu_h \rest C, 
\end{equation}
\begin{equation}\label{div-on-level-set-sing-term}
\dive (u b \sigma_h \rest C) = 0. 
\end{equation}
\end{subequations}
\end{lemma}
The proof of Lemma \ref{separation} would be fairly easy in the case when $\gamma$ is a straight line.
Roughly saying, in this case \eqref{div-on-cc} would read as 
\begin{equation*}
\int u(x) c_h(x) b(x) \psi'(x) \, dx + \int u(x) c_h(x) b(x) \psi'(x) \, d\sigma_h(x) + \int \psi(x) \, d\mu(x)=0,
\end{equation*}
$\psi \in C_0^\infty(\R)$.
Since $\sigma_h$ is concentrated on a $\L^1$-negligible set $S$, any $\phi \in C_0^1$ 
can be approximated in $C^0$-norm with a sequence of $C^1$-functions $\phi_n$ having $0$-derivative on $S$.
Consequently, $\phi_n'$ converge to $\phi'$ weak* in $L^\infty$ as $n \to \infty$.
Then, substituting $\psi = \phi_n$ and passing to the limit as $n\to \infty$ we get
\begin{equation*}
\int u(x) c_h(x) b(x) \phi'(x) \, dx + \int \phi(x) \, d\mu(x)=0.
\end{equation*}
Hence the only technicality here is to repeat this argument on a curve.

Before presenting the formal proof of Lemma~\ref{separation} we would like to discuss the parametric version of the equation \eqref{div-on-level-set-reg-term}.

Let $\gamma\colon I\to \R^2$ be an injective Lipschitz parametrization of $C$, where $I=\nicefrac{\R}{\ell \Z}$ or $I=(0,\ell)$ for some $\ell>0$
is the domain of $\gamma$. In view of Remark~\ref{orientation-by-b}) we can assume that the directions of $b$ and $\nabla^\perp H$ agree $\H^1$-a.e. on $C$. 
So there exists a constant $\varpi\in \{+1,-1\}$ such that
\begin{equation}\label{admissible-param}
\frac{b(\gamma(s))}{\vert b(\gamma(s)) \vert} = \varpi \frac{\gamma^{\prime}(s)} {\vert \gamma^{\prime}(s)\vert}
\end{equation}
for a.e. $s\in I$.
We will say that $\gamma$ is an \emph{admissible parametrization} of $C$ if $\varpi=+1$.
In the rest of the text we will consider only admissible parametrizations of the connected components $C$.

\begin{lemma}\label{param-lemma}
Equation \eqref{div-on-level-set-reg-term} holds iff for any admissible parametrization $\gamma$ of $C$
\begin{equation}\label{eq:div-ub-in-param-in-balls}
\partial_s (\hat{u}\hat{c}_h \vert \hat{b} \vert) = \hat{\mu}_{h}
\end{equation}
where $\gamma_\# \hat{\mu}_{h} = \mu_{h} \rest C$, $\hat{u} = u \circ \gamma$, $\hat{c}_{h}=c_{h} \circ \gamma$ and $\hat{b}=b \circ \gamma$.
\end{lemma}

In the proof of Lemma~\ref{param-lemma} we will use the following result:
\begin{lemma}[{\cite[Section 7]{ABC2}}] \label{lemma:ABC2_density} Let $a \in L^{1}(I)$ and $\mu$ a Radon measure on $I$, where $I=\nicefrac{\R}{\ell \Z}$ or $I=(0,\ell)$ for some $\ell>0$. Suppose that $\gamma \colon I \to \Omega$ is an injective Lipschitz function such that $\gamma^{\prime} \ne 0$ a.e. on $I$ and $\gamma(0,\ell) \subset \Omega$. Consider the functional
\begin{equation*}
\Lambda(\phi) := \int_{I} \phi^{\prime}a\, dt + \int_{I} \phi \, d\mu, \qquad \forall \phi \in \Lip_{c}(I).
\end{equation*}
If $\Lambda(\varphi \circ \gamma) = 0$ for any $\varphi \in C_{c}^{\infty}(\Omega)$ then $\Lambda(\phi) = 0$ for any $\phi \in \Lip_{c}(I)$. 
\end{lemma}

\begin{proof}[Proof of Lemma~\ref{param-lemma}] Let us recall a corollary from Area formula: if $\gamma \colon I \to \R^2$ is an injective Lipschitz parametrization of $C$ then
\begin{equation*} 
\H^{1} \rest C = \gamma_{\#}\left(\vert \gamma^{\prime} \vert \L^{1}\right).
\end{equation*}
Using this formula the distributional version of \eqref{div-on-level-set-reg-term},
\begin{equation*} 
\int_{C} u c_h b \cdot \nabla \phi d \H^1 \rest C +\int_{C} \phi d\mu_h = 0, \quad \forall \phi \in C_{c}^{\infty}(B),
\end{equation*}
can be written as 
\begin{equation*}
\int_{I} u(\gamma(s)) c_{h}(\gamma(s)) b(\gamma(s)) \cdot (\nabla \phi)(\gamma(s)) \vert \gamma^{\prime}(s) \vert \, ds + \int_{I} \phi(\gamma(s)) d\hat{\mu}_{h}(s) = 0 
\end{equation*}
where $\hat{\mu}_{h}$ is defined by $\hat{\mu}_{h} := \left( \gamma^{-1}\right)_{\#} \mu_{h}$.

Using \eqref{admissible-param} we can write the equation above as
\begin{equation*}
\int_{I} u(\gamma(s)) c_{h}(\gamma(s)) \gamma^{\prime}(s) (\nabla \phi)(\gamma(s)) \vert b(\gamma(s)) \vert \, ds + \int_{I} \phi(\gamma(s)) d\hat{\mu}_{h}(s) = 0,
\end{equation*}
which reads as 
\begin{equation*}
\int_{I} \hat{u}(s) \hat{c_{h}}(s) \partial_{s}\phi(\gamma(s)) \vert \hat{b}(s) \vert \, ds + \int_{I} \phi(\gamma(s)) d\hat{\mu}_{h}(s) = 0.
\end{equation*}
Since the equation above holds for any $\phi \in C_{c}^{\infty}(B)$ it remains to apply Lemma~\ref{lemma:ABC2_density}.
\end{proof}

\begin{proof}[Proof of Lemma~\ref{separation}]
Let us write $\Lambda(\phi) = M(\phi) + N(\phi)$, where 
\begin{equation*}
M(\phi) := \int_{C} u c_{h} (b\cdot \nabla \phi) \, d \H^{1} + \int_{C} \phi \, d\mu_{h}
\end{equation*}
and
\begin{equation*}
N(\phi) := \int_{C}ub \cdot \nabla \phi d\sigma_{h}
\end{equation*}
for every $\phi \in C_{c}^{\infty}(B)$.

Fix a test function $\phi$: we are going to ``perturb'' $\phi$ in such a way that $N(\phi)$ becomes arbitrarily small and $M(\phi)$ remains almost unchanged. Since $\Lambda(\phi)=0$ we will obtain that $\vert M(\phi) \vert < \e$ and this will imply that $M(\phi)=N(\phi)=0$.

By Lemma~\ref{disint-in-ball}, we have $\sigma_h \perp \H^1 \rest C$ therefore there exists a $\H^1$-negligible set $S\subset C$ such that $\sigma_h$ is concentrated on $S$. Moreover, by inner regularity, for every $n \in \N$, we can find a compact $K\subset S$ such that
\begin{equation*}
\sigma_h(S \setminus K) < \frac{1}{n}.
\end{equation*}
Using the fact that $\H^1(K)=0$, for every $n \in\N$, we can find countably many open balls $\{B_{r_j}(z_j)\}_{j \in \N}$ which cover $K$ and whose radii $r_{j}$ satisfy 
\begin{equation*}
\sum_{j\in \N} r_j < \frac{1}{n}. 
\end{equation*}
Furthermore, by compactness, we can extract from $\{B_{r_j}(z_j)\}_{l\in \N}$ a finite subcovering, $\{B_{r_j}(z_j)\}$ with $j=1,\ldots , \nu$ where $\nu = \nu(n) \in \N$ (we stress that $\nu$ depends on $n$).

For every $j \in \{1, \ldots, \nu \}$, let 
\begin{equation*}
P_{i}^{j,n} := (z_{j,i} - r_j, z_{j,i} + r_j)
\end{equation*}
denote the projection of $B_{r_{j}}(z_{j})$ onto the $x_{i}$-axis, with $i=1,2$. Since $P_{i}^{j,n}$ is an open interval we can find a smooth function $\psi_{i}^{j,n} \colon \R \to \R$ such that 
\begin{equation*}
\psi_{i}^{j,n}(\xi) =
\begin{cases}
0 & \xi \in P_{i}^{j,n}, \\
1 & {\rm dist} (\xi, \partial P_{i}^{j,n})>2r_{i},
\end{cases} 
\end{equation*}
and $0\le \psi_{i}^{j,n} \le 1$ for every $\xi \in \R$. Now we consider the product $\psi^n_i:=\psi^{1,n}_{i}\psi^{2,n}_{i} \cdots \psi^{\nu,n}_{i}$ and we define the functions $\chi^n_i \colon \R \to \R$ as
\begin{equation*}
\chi^n_i(\xi) := \int_{0}^{\xi} \psi^n_i(w) \, dw 
\end{equation*}
for $i=1,2$ and $n \in \N$. Now we set $\chi^{n}(x):= (\chi^n_1(x), \chi^n_2(x))$ and $\phi_{n}:=\phi \circ \chi^{n}$. Since $\Vert \chi^{n} - \text{id} \Vert_{\infty} \le 4 \sum_{i} r_{i} \le \frac{4}{n}$ we deduce that $\phi_{n} \to \phi$ uniformly in $C$ because 
\begin{equation*}
\vert \phi_{n}(x) - \phi(x) \vert \le \Vert \nabla \phi \Vert_{\infty} \Vert \chi^{n} - \text{id} \Vert_{\infty} \to 0
\end{equation*}
as $n \to +\infty$.

Let us now take an admissible parametrization of $C$, $\gamma \colon I \to \R$, and let us  introduce the functions $\hat{\phi}_{n} := \phi_{n} \circ \gamma$. Using for instance the density of $C^{1}$ functions in $L^{1}(I)$, we can actually show that $\partial_{s} \hat{\phi}_{n} \rightharpoonup^{\star} \partial_{s}\hat{\phi}$ in weak$^\star$ topology of $L^{\infty}$. Passing to the parametrization as in the proof of Lemma~\ref{param-lemma} we get
\begin{equation*}
 \int_{C} u c_{h} (b \cdot \nabla \phi_{n}) \, d \H^{1} = \int_{I} \hat{u} \hat{c}_{h} \hat{b} \, \partial_{s} \hat{\phi}_{n} \, ds,
\end{equation*}
where we denote by $\hat \cdot$ the composition with $\gamma$.

Using weak$^\star$ convergence, we obtain that 
\begin{equation*}
\int_{C} u c_{h} (b \cdot \nabla \phi_{n}) \, d \H^{1}  \to \int_{C} u c_{h} (b\cdot \nabla \phi) \, d \H^{1}.
\end{equation*}
On the other hand, by uniform convergence, we immediately get 
\begin{equation*}
 \int \phi_{n} \, d\mu_{h} \to \int \phi \, d\mu_{h}, 
\end{equation*}
as $n \to+\infty$. In particular, we have that $M(\phi_{n}) \to M(\phi)$.

Now observe that $\nabla \phi_{n} = 0$ on $K$ by construction, hence we get   
\begin{equation*}
 N(\phi_{n}) \le \int_{S \setminus K} \vert u b \vert \vert \nabla \phi_{n}\vert d\sigma_{h} \le \Vert ub\Vert_{\infty} \Vert \nabla \phi \Vert_{\infty} \frac{1}{n} \to 0
\end{equation*}
and this implies that $N(\phi)=0$. Therefore, $0=\Lambda(\phi)=M(\phi)$, which concludes the proof.
\end{proof}

We note, in particular, that from \eqref{div-on-level-set-sing-term}, being $b \in \BV$ and taking $u\equiv 1$ in \eqref{eq:div-ub-first-time}, we have that $\dive(b \sigma_{h} \rest E_{h}) = 0$ for a.e. $h$. 

Let
\begin{equation}
\label{E_def_F}
F:=\left\{b \ne 0, r_{B} = 0\right\} \cap E.
\end{equation}
By Point (5) of Lemma \ref{disint-in-ball}, $\sigma_{h}$ is concentrated on $F \cap E_{h}$ hence we have
\begin{equation}\label{sing_part}
\dive (\1_{F}b \sigma_{h}) = 0, \qquad \text{ for } \L^{1} \text{-a.e. } h.
\end{equation}
This important piece of information is very useful to prove the following 

\begin{lemma}
\label{lemma:divrf}
We have $\dive (\1_F b) = 0$ in $\ss D^{\prime}(B)$.
\end{lemma}

\begin{proof} For every test function $\phi\in C_{c}^{\infty}(B)$, we have
\begin{equation*}
\int_{F} (b(x) \nabla \phi(x)) \, dx = \int\int_{F \cap E_{h}} (b(x) \cdot \nabla \phi(x)) \, d\sigma_{h}(x)dh.
\end{equation*}
Using again Point (5) of Lemma \ref{disint-in-ball} and \eqref{sing_part}, we get that 
\begin{equation*}
\int_{F \cap E_{h}} (b(x) \nabla \phi(x)) \, d\sigma_{h}(x) = 0
\end{equation*}
and then we conclude.
\end{proof}

Finally, let us mention a covering property of the set $E^\star$:

\begin{lemma}\label{l-E-covers-nonzero-b}
Let $E^\star$ be the set defined in \eqref{Ndef}. Then 
\begin{equation*}
E^\star \supset \{\nabla H \ne 0\} \mod \L^2.
\end{equation*}
\end{lemma}

\begin{proof}
Suppose that $P:= \{\nabla H \ne 0\} \setminus E$ has positive measure. Then
\begin{equation*}
0< \int_P |\nabla H| \, dx = \int \int \1_P \,d\H^1 \rest E_h \, dh = 0
\end{equation*}
where the first equality is due to Coarea formula (Lemma~\ref{lemma-coarea})
and the second equality holds since $\1_P$ is zero on $E_h$ for a.e. $h$.
\end{proof}
Note that in general $E^\star$ can contain a subset of $\{\nabla H =0\}$ with positive measure (see \cite{ABC2}).
However, in the next section we show that, if $H$ has the so-called \emph{weak Sard property}, then in fact $E^\star = \{\nabla H \ne 0\} \mod \L^2$.

\section{Weak Sard Property of Hamiltonians}
\label{s:new-four}

\subsection{Matching properties}

As we have seen at the beginning of Section \ref{ss:disintegration-lebesgue}, to every Hamiltonian $H$ we can associate a triple $(H, N, E)$ where $N$ is the set given by Theorem \ref{thm:ABC2_structure} and $E= \cup_{h \notin N} E_{h}^{\star}$.

Suppose now we have another triple $(\tilde H, \tilde N, \tilde E)$; we ask whether, given $x\in E \cap \tilde E$ it is true that $C_x = \tilde C_x$. This is essentially the definition of matching property; moreover, we will prove the ``Matching Lemma'', which states that gradients of $H$ and $\tilde{H}$ being parallel (in a simply connected set) is a sufficient condition for matching.

\subsection{Matching of two Hamiltonians}

Let us consider two Lipschitz Hamiltonians $H_{1}$ and $H_{2}$, defined on the same open, simply connected set $A$; according to Theorem \ref{thm:ABC2_structure}, we have two negligible sets $N_{1}$ and $N_{2}$ such that the level sets $E^{1}_{h}$ and $E_{h^{\prime}}^{2}$ of $H_{1}$ and $H_{2}$ are regular for $h \notin N_{1}$ and $h^{\prime} \notin N_{2}$. 
We set $E_{1}:= \cup_{h \notin N_{1}} E^{1}_{h}$ and $E_{2}:= \cup_{h^{\prime} \notin N_{2}} E^{2}_{h^{\prime}}$.

\begin{definition}
The Hamiltonians $H_{1}$ and $H_{2}$ \emph{match} in an open subset $A^{\prime} \subset A$ if $C_x^1 = C_x^2$ for $\mathscr{L}^2$-a.e. $x \in A^{\prime} \cap E_{1} \cap E_{2}$, where $C_x^i$ denotes the connected component in $A^{\prime}$ of the level sets $H_i^{-1}(H_i(x))$ which contains $x$.
\end{definition}

As usual, given two vectors $a$ and $b$ in $\R^2$ we write $a \parallel b$ if $a = \alpha b$ or $b = \alpha a$ for some real number $\alpha$.

We now state and prove the following

\begin{lemma}[Matching lemma]\label{lemma:matching}
Let $H_{1}, H_{2}$ be defined as above. If $\nabla H_1 \parallel \nabla H_2$ a.e. on $A^{\prime} \subset A$ open, then the Hamiltonians $H_{1}$ and $H_{2}$ match in $A^{\prime}$.
\end{lemma}

\begin{proof}
Let $b_1:= \nabla^{\perp}H_1$. Then $\dive b_1 = 0$. Let us prove that 
\begin{equation}\label{E_const_level_sets}
\dive(H_2 b_1)=0
\end{equation}
in the sense of distributions. Indeed, we have for every $\varphi \in \text{Lip}_c(A^{\prime})$
\begin{equation*}
\int H_2 (b_1 \cdot \nabla \varphi) \, dx = \int \big[ b_1 \cdot \nabla(H_2 \varphi) - \varphi (b_1 \cdot \nabla H_2) \big] \, dx.
\end{equation*}
The first term is zero because $\dive b_1 = 0$ (and $\varphi H_2$ can be used as test function since it is Lipschitz); the second term is also zero because $\nabla H_2 \parallel \nabla H_1$ a.e. on $A'$, hence $b_1 \perp \nabla H_2$ a.e. on $A^\prime$.

From \eqref{E_const_level_sets}, using  \cite[Theorem 4.1 and 6.1]{BG}, we obtain that there exists a $\mathscr L^1$ negligible set $N$ such that $H_2$ is constant on every non trivial connected components $C \cap A'$ of the level sets of $H_1$ which do not correspond to values in $N$. By disintegration, we have that the sets of points $x \in A' \cap E_1$ such that $H_1(x) \notin N$ are a negligible set and therefore we can infer that for a.e. $x \in A' \cap E_1$, $H_{2}$ is constant along the connected components in $A^{\prime}$ of the level sets of $H_{1}$. By repeating the same argument for $H_2$ we get the claim.
\end{proof}

\subsection{The Weak Sard property}

Let $f \colon \R^2 \to \R$ be a Lipschitz function and let $S$ be the critical set of $f$, defined as the set of all $x \in \R^{2}$ where $f$ is not differentiable or $\nabla f(x)=0$. We are interested in the following property:

\smallskip

\noindent {\it the push-forward according to $f$  of the restriction of $\L^{2}$ to $S$ is singular with respect to $\L^{1}$, that is }
\begin{equation*}
f_{\#}\left(\L^{2} \rest S\right) \perp \L^{1}. 
\end{equation*}

\smallskip

\noindent This property clearly implies the following \emph{Weak Sard Property}, which is used in \cite[Section 2.13]{ABC1}:
\begin{equation*}
 f_{\#}\left( \L^{2} \rest (S \cap E^{\star}) \right) \perp \L^{1},
 \end{equation*}
where the set $E^{\star}$ is the union of all connected components with positive length of all level sets of $f$. We point out that the relevance of the Weak Sard Property in the framework of transport and continuity equation is explained is \cite[Theorem 4.7]{ABC1}.

\begin{remark}\label{l-wsp-criterea}
Informally, the weak Sard property means that the ``good'' level sets of $H$ do not intersect the critical set $S$, apart from a negligible set.
In terms of the disintegration of the Lebesgue measure \eqref{disint-lebesgue}, we can say that $H$ has the weak Sard property if and only if $\sigma_h=0$ for a.e. $h$.
\end{remark}

Now we give the following 

\begin{definition} We set 
 \begin{equation*}
\tilde{r}_{B}:=r_{B} + \1_{F},
\end{equation*}
where we recall that $r_B$ is the function defined in \eqref{E_def_r_B} and $F$ is the set defined in \eqref{E_def_F}.
\end{definition}

By linearity of divergence, by Lemma \ref{lemma:divrb} and Lemma \ref{lemma:divrf}, we have 
\begin{equation*}
\dive (\tilde{r}_{B}b)=0 
\end{equation*}
in $\ss D^{\prime}(B)$. Therefore, we conclude that there exists a Lipschitz potential $\tilde{H}$ such that $\nabla \tilde{H}^{\perp} =\tilde{r}_{B}b$.

Moreover, we observe that $\nabla H \parallel \nabla \tilde{H}$ a.e. in $B$: therefore we can apply Matching Lemma \ref{lemma:matching} to get that the regular level sets of $H$ and of $\tilde{H}$ agree. In particular, we obtain $E = \tilde{E} \mod \L^{2}$, directly from the definition of $\tilde{H}$. 
We note also that the function $\tilde{H}$ has the Weak Sard property: indeed, directly from the construction, we have $\nabla \tilde{H} \ne 0$ on $E$ hence, since $E=\tilde{E}\mod \L^{2}$, it follows that $\L^{2}(\tilde{E} \cap \tilde{S})=0$.

Finally, disintegrating $\L^{2} \rest E$ with respect to $H$ we get 
\begin{equation*}
\L^{2} \rest E = \int_{\R} (c_{h} \H^{1} \rest E_{h} + \sigma_{h})\,dh, 
\end{equation*}
while using the Hamiltonian $\tilde{H}$
\begin{equation*}
\L^{2} \rest E = \int_{\R} \tilde{c}_{h} \H^{1} \rest \tilde{E}_{h}\,dh.
\end{equation*}

In particular, it follows that $\sigma_{h}=0$ for a.e. $h$, which means that $H=\tilde{H}$ (up to additive constants) and $H$ has the Weak Sard Property.

We collect this result in the following
\begin{lemma}
\label{L_weak_Sard_H}
The Hamiltonian $H_B$ has the weak Sard property.
\end{lemma}

We conclude this section with the following corollary concerning the covering properties of the set $E^\star$ defined in \eqref{Ndef}:
\begin{corollary}\label{c-weak-Sard}
Suppose that $H$ has the weak Sard property.
Let $E^\star$ be the set defined in \eqref{Ndef}. Then
\begin{equation*}
E^\star = \{\nabla H \ne 0\} \mod \L^2.
\end{equation*}
\end{corollary}

\begin{proof}
The argument is similar to Lemma~\ref{l-E-covers-nonzero-b}.
Let $Q = E^\star \setminus \{\nabla H\ne 0\}$. By~\eqref{disint-lebesgue}
\begin{equation*}
\L^2(Q) = \int \left(\int_Q \, d\sigma_h \right) \, dh = 0
\end{equation*}
since by Remark~\ref{l-wsp-criterea} $\sigma_h=0$ for a.e. $h$.
\end{proof}

\begin{remark}\label{rem:referee} If we do not assume $\BV$ regularity of $b$, but $b(x)\ne 0$ for $\L^2$-a.e. $x\in \R^2$ the conclusion of Lemma \ref{L_weak_Sard_H} still holds. This can be proved using minor modifications of the above argument. More precisely, since $b$ is nearly incompressible the function $m(x):=\int_0^T \rho(\tau,x) \, d\tau$, where $\rho$ is the density of $b$, solves
\begin{equation} \label{e-continuity-integrated-in-time}
\dive (m b) = \rho(T, \cdot) - \rho(0, \cdot)
\end{equation}
in $\ss D'(B)$, being $\rho(T,\cdot)$ and $\rho(0,\cdot)$ the weak-$\star$ limits in $L^\infty$ of $\rho(t,\cdot)$ as $t\to T$ and $t\to 0$ respectively.
Applying Lemmas~\ref{disint-div-ub}, \ref{cc-lemma}, \ref{separation} with $u=m$, from \eqref{div-on-level-set-sing-term} we obtain
\begin{equation}\label{e-m-sigma}
\dive (m b \sigma_h \rest C) = 0.
\end{equation}
Hence Lemma \ref{lemma:divrf} holds replacing $\1_F b$ with $m \1_F b$: in particular, setting 
\begin{equation*}
\tilde{r}_{B}:=r_{B} + m\1_{F}
\end{equation*}
we can repeat the argument of Section \ref{s:new-four}. 
\end{remark}

\section{Level sets and trajectories I}\label{s:new-five}

In this section, we assume that $H_{B}$ is defined on all $\R^2$ (using standard theorems for the extension of Lipschitz maps).

\subsection{Trajectories} We now present some lemmas which relate the trajectories $\gamma \in \mathsf T_{B}$ to the level sets of the Hamiltonian. The first result we prove is that $\eta$-a.e. $\gamma$ is contained in a level set. 

\begin{lemma} \label{lemma:library} Let $B \in\mathscr B$, $t_{1},t_{2} \in [0,T]$ and set $\mathsf T:= \left\{\gamma: \gamma\left( (t_{1}, t_{2}) \right) \subset B \right\}$. Then $\eta$-a.e. $\gamma \in \mathsf T$ we have $(t_{1}, t_{2}) \ni t \mapsto H(\gamma(t))$ is a constant function.\end{lemma}

\begin{proof} Let $(\varrho_{\varepsilon})_{\varepsilon}$ be the standard family of convolution kernels in $\R^{2}$. We set $H_{\varepsilon}(x):= H \star \varrho_{\varepsilon}(x)$ for any $x\in B$.

For every $t \in [t_{1}, t_{2}]$ define 
\begin{equation*}
I(t):= \int_{\mathsf T} \vert H(\gamma(t))-H(\gamma(0)) \vert d\eta(\gamma) 
\end{equation*}
and we will prove $I \equiv 0$. 

First note that $I$ is positive because the integrand is non-negative and $\eta$ is positive. On the other hand,  
\begin{equation*}
\begin{split}
 I(t) \le  \underbrace{ \int_{\mathsf T} \vert H(\gamma(t))-H_{\eps}(\gamma(t)) \vert d\eta(\gamma)}_{ I^{\eps}_1} + \underbrace{  \int_{\mathsf T} \vert H_{\eps}(\gamma(t))-H_{\eps}(\gamma(0)) \vert d\eta(\gamma) }_{I^{\eps}_2} \\ +
 \underbrace{ \int_{\mathsf T} \vert H_{\eps}(\gamma(0))-H(\gamma(0)) \vert d\eta(\gamma)}_{I^{\eps}_3}.
 \end{split}
\end{equation*}
Now for a.e. $x \in \R^2$ we have $H_{\eps}(x) \to H(x)$: hence 
\begin{equation*}
\int_{\mathsf T} \left\vert H_{\eps}(\gamma(t)) - H(\gamma(t)) \right\vert d\eta(\gamma) \le \int_{B} \vert H_{\eps}(x) - H(x)\vert \rho(t,x)dx \to 0
\end{equation*}
as $\eps \to 0$. Therefore, we can infer that
\begin{equation*}
I^{\eps}_1 \to 0,  \qquad I^{\eps}_3 \to 0 
\end{equation*}
as $\eps \downarrow 0$.

Let us study $I^{\eps}_2$. We have
\begin{equation*}
\begin{split}
I^{\eps}_2(t) & \le \int_{\mathsf T} \int_{t_{1}}^{t} \vert \partial_s H_\eps(\gamma(s)) \vert \, ds \, d\eta(\gamma) \\
                 &= \int_{\mathsf T} \int_{t_{1}}^{t}  \vert \nabla H_{\eps}(\gamma(s)) \cdot b(\gamma(s)) \vert \, ds \, d\eta(\gamma) \\ 
                 & = \int_{t_{1}}^{t} \int \vert  \nabla H_{\eps}(x) \cdot b(x) \vert \, d(e_t \# \eta \rest\mathsf T)(x) \, ds \\ 
                 & \le \int_{0}^{T} \int \vert\nabla H_{\eps}(x) \cdot b(x) \vert \rho_{\mathsf T}(t,x)\, dx \, ds \\
                 & = \int \vert \nabla H_{\eps}(x) \cdot b(x) \vert r_{\mathsf T}(x) \, dx \to \int \vert \nabla H(x) \cdot b(x) \vert r_{\mathsf T}(x) \, dx = 0  
\end{split}
\end{equation*}
where we have used $\nabla H_{\eps}(x) \to \nabla H(x)$ for a.e. $x$. In the end, we have that $I_2^\eps \to 0$ as $\eps \downarrow 0$ and this concludes the proof. \end{proof}

We now show that Lemma \ref{lemma:library} can be improved, showing indeed that $\eta_{B}$-a.e. $\gamma$ is contained in a \emph{regular} level set of $H$. 

\begin{lemma}\label{lemma:library-plus} Up to a $\eta_B$ negligible set, the image of every $\gamma \in \mathsf T_{B}$ is contained in a connected component of a \emph{regular} level set of $H_B$.
\end{lemma}
\begin{proof}
Using Lemma \ref{lemma:library}, we remove $\eta_B$-negligible set of trajectories along which $H_B$ is not constant. Set $E^{c}:=B \setminus E$ and consider the set 
\begin{equation*}
\mathscr P:= \big\{\gamma \in \mathsf T_{B}: \, \gamma\left( (0,T) \right) \cap B \subset E^{c} \big\}. 
\end{equation*}
It is enough to show that $\eta(\mathscr P) = 0$: this means that for $\eta$-a.e. $\gamma$ the image $\gamma(0,T)$ is not contained in the complement of $E$ and thus 
we must have (in the ball) $\gamma(0,T) \subset E$ for $\eta$-a.e. $\gamma \in \mathsf T_{B}$ (this follows remembering that a.e. $\gamma$ is contained in a level set).

By Coarea formula (see Lemma \ref{disint-in-ball}), $\vert \nabla H\vert \mathscr{L}^2 \rest E^{c}  = 0$, i.e. 
\begin{equation*}
 \int \1_{E^{c}}(x) \vert \nabla H(x) \vert\, dx = 0.
\end{equation*}
Since $\nabla H = r_B b^{\perp}$ in $B$ and $r_B \ge 0$ (since $\rho_{B} > 0$), we have
\begin{equation*}
\begin{split}
0 & = \int \1_{E^{c}}(x) \vert r_B(x) b(x) \vert \, dx  \\
& = \int \1_{E^{c}}(x) r_{B}(x) \vert b(x) \vert \, dx \\
& = \int \int_0^T \1_{E^{c}}(x) \rho_B(t,x) \vert b(x) \vert \, dx \, dt. 
 \end{split}
\end{equation*}
Using \eqref{eq:rho-e-eta} we have
\begin{equation*}
0 = \int_0^T \int \1_{E^{c}}(\gamma(t)) \vert b(\gamma(t)) \vert \, d\eta(\gamma) \, dt = \int_0^T \int_{\mathscr P} \vert b(\gamma(t)) \vert \, d\eta(\gamma) \, dt 
\end{equation*}
which implies (by Fubini) that for $\eta$-a.e. $\gamma \in \mathscr P$ we have
\begin{equation*}
\int_0^T \vert b(\gamma(t))\vert  \, dt = 0.
\end{equation*}
This gives $\vert b(\gamma(t))\vert =0$ for a.e. $t\in [0,T]$ and this contradicts the definition of $\mathsf T_{B}$. Hence $\eta(\mathscr P)=0$.
\end{proof}

\section{Locality of the divergence}
\label{s:new-six}

In this section we prove that the if $\dive(ub)$ is a measure, then it is $0$ on the set 
\begin{equation}\label{eq:set_M}
M:= \bigg\{x \in \R^2: b(x) = 0,\, x \in \mathcal D_{b} \text{ and } \nabla^{\rm appr} b(x) = 0 \bigg\},
\end{equation}
where $\mathcal D_{b}$ is the set of approximate differentiability points and $\nabla^{\rm appr} b$ is the approximate differential, according to Definition \cite[Def. 3.70]{AFP}. For shortness, we will call this property \emph{locality of the divergence}.

Let $U$ be an open set in $\R^d$, $d\in \N$.
The main result of this section is the following

\begin{proposition}\label{prop:locality} $u \in L^{\infty}(U)$ \text{and} suppose that $\dive (ub) =\lambda $ in the sense of distributions, where $\lambda$ is a Radon measure on $U$. Then $\vert \lambda \vert \rest M = 0$. 
\end{proposition}

Note that we do not assume any weak differentiability of $u$ or $ub$, so the conclusion of Proposition~\ref{prop:locality}
does not follow immediately from the standard locality properties of the approximate derivative (see e.g. \cite{AFP}, Proposition~3.73). Moreover, we also mention a related counterexample (contained in \cite{ABC2}), where the authors construct a bounded vector field $V$ on the plane whose (distributional) divergence belongs to $L^\infty$, is non-trivial, and is supported in the set where $V$ vanishes.
Our proof is based on Besicovitch-Vitali covering Lemma (\cite[Thm. 2.19]{AFP}) and uses some basic facts about the trace properties of $L^{\infty}$ vector fields whose divergence is a measure (we refer to \cite{ChenFrid,ambcriman04} or \cite{camillonote}). In particular, we recall the following Theorem (for the proof, see \cite[Prop 7.10]{camillonote}):

\begin{theorem}[Fubini's Theorem for traces] Let $\Omega \subset \R^d$ be an open set and $B \in L^{\infty}_{\rm{loc}}(\Omega, \R^{d})$ be a vector field whose distributional divergence $\dive B =: \mu$ is a Radon measure with locally finite variation in $\Omega$. Let $F \in C^{1}(\Omega)$. Then for a.e. $t \in \R$ we have 
\begin{equation}\label{eq:fubini_traces} 
\text{Tr}(B, \partial \{F > t \}) = B \cdot \nu \qquad \H^{d-1}\text{-a.e. on } \Omega \cap \partial \{F > t \}, 
\end{equation}
 where $\nu$ denotes the exterior unit normal to $\partial \{F>t\}$ and the distribution $\text{Tr}(B, \partial \Omega^{\prime})$ is defined by 
\begin{equation*}
\langle \text{Tr}(B, \partial \Omega^{\prime}), \phi \rangle := \int_{ \Omega^{\prime}} \phi \, d \mu + \int_{\Omega^{\prime}} \nabla \phi \cdot B \, dx , \qquad \forall \phi \in C_{c}^{\infty}(\Omega).  
\end{equation*}
for every open subset $\Omega^{\prime} \subset \Omega$ with $C^{1}$ boundary. \end{theorem}

Furthermore, we will use the following elementary
\begin{lemma}\label{lemma:camillo_192} Let $G \colon \R^{d} \to \R$ be a bounded, Borel function. For every $r>0$ there exists a set of positive measure of real numbers $s=s(r) \in [r,2r]$ such that
\begin{equation*}
\int_{\partial B_{s(r)}} \vert G(x)\vert \, d\H^{d-1}(x) \le \frac{1}{r} \int_{B_{2r}} \vert G(y) \vert \,dy.
\end{equation*}
\end{lemma}

\begin{proof}[Proof of Proposition \ref{prop:locality}]
Fix an arbitrary $x \in M$. For brevity let $B_{r} :=B_{r}(x)$. By \eqref{eq:fubini_traces} with $F(y):= \vert x - y \vert^{2}$,
there exists an $\L^1$-negligible set $N_x$ such that for any positive number $r\notin N_x$ we have
\begin{equation*}
\vert \lambda (B_{r}) \vert = \left\vert \int_{\partial B_{r}} u b \cdot \nu \, d\H^{d-1} \right\vert \le C  \int_{\partial B_{r}} \vert b \vert \, d\H^{d-1},
\end{equation*}
where $\nu$ denotes the exterior unit normal to $\partial B_{r}$.
By Lemma \ref{lemma:camillo_192}
\begin{equation*}
C \int_{\partial B_{r}} \vert b \vert \, d\H^{d-1} \le \frac{C}{r}\int_{B_{2r}} \vert b(x) \vert \, dx = o(r^{d})
\end{equation*} 
because, by definition of $M$, we have $\fint_{B_{r}} \vert b\vert \, dx = o(r)$. Therefore
\begin{equation}\label{eq:locality-estimate-old}
 \vert \lambda (B_{r}) \vert = o(r^{d}).
\end{equation}

Fix $\e > 0$. By \eqref{eq:locality-estimate-old} for any $x\in M$ there exists $\delta_x>0$ such that for any positive number $r<\delta_x$ such that $r\notin N_x$ we have
\begin{equation}\label{eq:locality-estimate}
 \vert \lambda (B_{r}(x)) \vert \le \e r^{d}.
\end{equation}

Let $S \subset M$ be an arbitrary bounded subset.

By regularity of $\lambda$, there exists a bounded open set $O \supset S$ such that $\vert \lambda \vert (O \setminus S) < \e$. Hence, for any $x \in S$ there exists $\rho_{x}>0$ such that $B(x, r) \subset O$ for any positive number $r < \rho_{x}$.
Consequently
\begin{equation*}
\mathscr F := \big\{ B(x,r): x \in S, r < \min(\rho_x, \delta_x), r \notin N_x \big\}
\end{equation*}
is a fine covering of $S$.

Hence we can apply Besicovitch-Vitali covering Lemma (\cite[Thm. 2.19]{AFP}): there exists a countable disjoint subfamily $\{B_{i}\}_{i \in \N} \subset \mathscr F$ such that 
\begin{equation*}
 \vert \lambda \vert \left( S \setminus \bigcup_{i} B_{i} \right) = 0.
\end{equation*}
 On the other hand, since $\bigcup_{i} B_{i} \subset O$ by construction, we have 
\begin{equation*}
 \vert \lambda \vert \left( \bigcup_{i} B_{i} \setminus S \right) < \e.
\end{equation*}

Using \eqref{eq:locality-estimate}, since the balls $B_{i}$ are disjoint, we have 
\begin{equation*}
\lambda\left( \bigcup_{i} B_{i} \right) = \sum_{i} \lambda(B_{i}) \le \e \L^{2}\left( \bigcup_{i} B_{i} \right). 
\end{equation*}
Hence
\begin{equation*}
 \lambda(S) = \lambda \left( \bigcup_{i} B_{i} \right) - \lambda \left( \bigcup_{i} B_{i} \setminus S\right) \to 0
\end{equation*}
as $\e \downarrow 0$. Hence $\lambda \rest S=0$ and, by arbitrariness of $S \subset M$, $\lambda \rest M =0$.
\end{proof}

\subsection{Comparison between $\L^{2}$ and $\eta$}\label{s:new-seven}

We present here two general lemmas which relate the Lebesgue measure $\L^{2}$ and the measure $\eta$ and are based on nearly incompressibility of the vector field $b$.

\begin{lemma} \label{lemma:eta_lebesgue}Let $A \subset \R^2$ be a measurable set. Then $\L^{2}(A)=0$ if and only if $\eta(\Gamma_{A})=0$ where
\begin{equation*}
\Gamma_{A} := \left\{\gamma \in \Gamma: \L^{1}(\{ t \in [0,T]: \gamma(t) \in A\})>0\right\}. 
\end{equation*}
\end{lemma}

\begin{proof} Let us prove first that $\L^{2}(A)=0$ implies $\eta(\Gamma_{A})=0$. We denote by $\rho_{A}$ the density such that $\rho_{A}(t,\cdot)\L^{2} = {e_{t}}_{\#} \left(\eta \rest \Gamma_{A}\right)$ and $r_{A}(x):= \int_{0}^{T} \rho_{A}(t,x)\,dt$. We have, using Fubini,
\begin{equation*}
\begin{split}
0 & = \L^{2}(A) = r_{A}\L^{2}(A) = \int_{0}^{T}\int_{\Gamma} \1_{A}(x)\rho_{A}(t,x) \, dx\,dt \\ 
   & = \int_{0}^{T} \int_{\Gamma} \1_{A}(\gamma(t)) \, d\eta(\gamma) \, dt \\
   & = \int_{\Gamma}\int_{0}^{T} \1_{A}(\gamma(t)) \, dt \, d\eta(\gamma) \\
   & = \int_{\Gamma_{A}} \int_{0}^{T} \1_{A}(\gamma(t)) \, dt \, d\eta(\gamma) \\
   & = \int_{\Gamma_{A}} \L^{1} \big( \{ t \in [0,T]: \, \gamma(t) \in A\} \big) \, d\eta(\gamma),
\end{split}
\end{equation*}
hence, $\L^{1}( \{ t \in [0,T]: \, \gamma(t) \in A\} )=0$ for $\eta$-a.e. $\gamma \in \Gamma_{A}$.

For the opposite direction, using that $\rho$ is uniformly bounded from below by $1/C$, we get
\begin{equation*}
\begin{split}
\frac{T}{C}\L^{2}(A) & = \frac{T}{C}\int \1_{A}(x)\,dx = \frac{1}{C}\int_{0}^{T} \int \1_{A}(x)\,dx \,dt  \\ 
   & \le \int_{0}^{T}\int \1_{A}(x)\rho(t,x) \, dx\,dt \\ 
   & = \int_{0}^{T} \int_{\Gamma} \1_{A}(\gamma(t)) \, d\eta(\gamma) \, dt \\
   & = \int_{\Gamma} \int_{0}^{T} \1_{A}(\gamma(t)) \, dt \, d\eta(\gamma) \\
   & = \int_{\Gamma_{A}} \int_{0}^{T} \1_{A}(\gamma(t)) \, dt \, d\eta(\gamma) \\
   & = \int_{\Gamma_{A}} \L^{1}( \{ t \in [0,T]: \, \gamma(t) \in A\} ) \, d\eta(\gamma) = 0. \qedhere
\end{split}
\end{equation*}
\end{proof}

\begin{lemma}\label{lemma:eta_lebesgue_2} We have $\L^{2}(A)=0$ if and only if $\eta(\Gamma^{s}_{A})=0$ for every $s \in [0,T]$.\end{lemma}

\begin{proof} For direct implication 
\begin{equation*}
\begin{split}
0 & = \L^{2}(A) = \int \1_{A}(x)\rho(s,x) \, dx \\ 
   & = \int_{\Gamma} \1_{A}(\gamma(s)) \, d\eta(\gamma) \\
   & = \int_{\Gamma^{s}_{A}} \1_{A}(\gamma(s))\, d\eta(\gamma) = \eta(\Gamma^{s}_{A}).
\end{split}
\end{equation*}
For the opposite direction,
\begin{equation*}
\begin{split}
\frac{1}{C}\L^{2}(A) & = \frac{1}{C}\int \1_{A}(x)\,dx \\ 
   & \le \int \1_{A}(x)\rho(s,x) \, dx \\ 
   & = \int_{\Gamma} \1_{A}(\gamma(s)) \, d\eta(\gamma) \\
   & = \int_{\Gamma^{s}_{A}} \1_{A}(\gamma(s))\, d\eta(\gamma) = \eta(\Gamma^{s}_{A})=0.
\end{split}
\end{equation*}\end{proof}

We now recall the set $M$, defined in \eqref{eq:set_M} as
\begin{equation*}
M:= \bigg\{x \in \R^2: b(x) = 0,\, x \in \mathcal D_{b} \text{ and } \nabla^{\rm appr} b(x) = 0 \bigg\},
\end{equation*}
and we consider the sets 
\begin{equation*}
\tilde{\Gamma}_{M} := \tilde{\Gamma} \cap \Gamma_{M}
\end{equation*}
and 
\begin{equation*}
\tilde{\Gamma}^{s}_{M} :=\left\{ \gamma \in \tilde{\Gamma}: \gamma(s) \in M \right\}. 
\end{equation*}
Using Proposition \ref{prop:locality}, we can show the following

\begin{lemma}\label{lemma:eta_M} Let $M$ be the set defined in \eqref{eq:set_M} and for every fixed $s \in [0,T]$ let $\tilde \Gamma_M^s := \{\gamma \in \tilde \Gamma \colon \gamma(s) \in M\}$. Then: 
\begin{itemize}
 \item $\eta(\tilde \Gamma_M^s)=0$ for a.e. $s \in [0,T]$; 
 \item $\eta(\tilde{\Gamma}_{M})=0$.
\end{itemize}
\end{lemma}

\begin{proof} 
Let us denote by $\eta^{s}_{M} := \eta \rest \tilde{\Gamma}_{M}^{s}$ and consider the Borel function
\begin{equation*}
\rho^s_{M}(t,\cdot)\L^{2} = {e_{t}}_{\#} \eta^{s}_{M}.
\end{equation*}
It is easy to see that $\rho^s_{M}$ solves continuity equation
\begin{equation}
\label{E_cont_rho_M}
\partial_{t} \rho^s_{M} + \dive (\rho^s_{M}b)=0.
\end{equation}
Integrating in time on $[0,t]$ we get 
\begin{equation*}
\dive \left(b  \int_{0}^{t} \rho^s_{M}(\tau , \cdot) d\tau \right)=(\rho^s_{M}(t, \cdot) - \rho^s_{M}(0, \cdot)) \L^{2}.
\end{equation*}
In particular, thanks to Proposition \ref{prop:locality}, we have that 
\begin{equation}
\label{E_const_rho_M}
\big( \rho^s_{M}(t, \cdot) - \rho^s_{M}(0, \cdot) \big) \L^{2} \rest M = 0,
\end{equation}
hence $\rho^s_{M}(t, \cdot) = \rho^s_{M}(0, \cdot)$, for a.e. $x$. Furthermore, integrating in space the continuity equation \eqref{E_cont_rho_M} we get the conservation of mass:
\begin{equation}\label{eq:cons_mass}
\frac{d}{dt} \int_{\R^2} \rho^s_{M}(t,x)\,dx = 0.
\end{equation}
Therefore, using \eqref{E_const_rho_M} and \eqref{eq:cons_mass}, we have 
\begin{equation*}
\begin{split}
\int_{\R^2 \setminus M} \rho^s_{M}(t,x) dx & =  \int_{\R^2} \rho^s_{M}(t,x) dx -  \int_{M} \rho^s_{M}(t,x) dx =  \\
& = \int_{\R^2} \rho^s_{M}(s,x) dx -  \int_{M} \rho^s_{M}(s,x) dx = \int_{\R^2 \setminus M} \rho^s_{M}(s,x) dx = \\
& = \int \1_{\R^2 \setminus M}(\gamma(s))d\eta_{M}(\gamma) = 0,
\end{split} 
\end{equation*}
which gives us $\rho^s_{M}(t, \cdot) = 0$ a.e. on $\R^2 \setminus M$. Hence 
\begin{equation*}
 0 =\int_{0}^{T} \int_{\R^2 \setminus M} \rho^s_{M}(t,x)\, dx =\int_{0}^{T} \int 1_{\R^2 \setminus M}(\gamma(t))\, d\eta^s_{M}(\gamma)\,dt
\end{equation*}
and this implies that $\eta^s_M(\tilde \Gamma^{s}_{M}) = 0$ for $s \in [0,T]$, since $\gamma \in \tilde{\Gamma}$ are not constant functions (by definition) and $b=0$ on $M$. 

Now the second part easily follows from the first one by a Fubini-like argument: indeed, we set 
\begin{equation*}
I:=\int_0^T \eta(\tilde \Gamma_M^s) \, ds = 0.
\end{equation*}
Since $\eta(\tilde \Gamma_M^s) = \int_{\tilde \Gamma} 1_M(\gamma(s)) \, d\eta(\gamma)$ and using Fubini's theorem we get
\begin{equation*}
I=\int_{\tilde \Gamma} \int_0^T 1_M(\gamma(s)) \, ds \, d\eta(\gamma) = 0
\end{equation*}
i.e. $\L^{1} (\{t \in [0,T]: \, \gamma(t) \in M \}) = 0$ for $\eta$-a.e. $\gamma \in \tilde{\Gamma}_{M}$ and this concludes the proof.\end{proof}

\section{Level sets and trajectories II}\label{s:new-seven}

The results obtained in the Section \ref{s:new-six} provide us with a better description of the relantionship between the trajectories $\gamma \in \Gamma_B$ and the level sets of $H_B$, thus improving the results of Section \ref{s:new-five}.

\subsection{Trajectories and level sets coincide up to a translation in time}

Let $B \in \mathscr B$ a fixed ball of the collection and, as usual, let $H_B$ denote its Hamiltonian. Thanks to Lemma \ref{lemma:library-plus}, there exists a $\eta$-negligible set $N$ such that for every $\gamma \in \Gamma_B \setminus N$ the image $\gamma(0,T)$ is contained in a connected component $\mathfrak c$ of a regular level set of $H_B$. 

Recalling \cite[Theorem 2.5(iv)]{ABC2}, there exists a parametrization $\gamma_{\mathfrak c}$ of $\mathfrak c$ with the following properties: 

\begin{itemize}
\item $\gamma_{\mathfrak c} \colon I_{\mathfrak c} \to \R^2$ is a Lipschitz map, where $I_{\mathfrak c}=\nicefrac{\R}{\ell \Z}$ or $I_{\mathfrak c}=[0,\ell]$ for some $\ell>0$ is the domain of $\gamma$; 
\item $\gamma_{\mathfrak c}$ is injective; 
\item $\gamma_{\mathfrak c}^\prime(s) = b(\gamma_{\mathfrak c}(s))$ for $\L^1$-a.e. $s \in I_{\mathfrak c}$. 
\end{itemize}

Thus it makes sense to wonder about the relationship between the trajectory $\gamma \in \Gamma_B \setminus N$ and the parametrization $\gamma_{\mathfrak c}$ of the corresponding connected component. The following proposition precises this relation, showing that $\gamma$ and $\gamma_{\mathfrak c}$ coincide up to a translation in time.

\begin{proposition}\label{prop:translation-in-time} Let $N$ be the set given by Lemma \ref{lemma:library-plus} and $\gamma \in \tilde{\Gamma} \setminus N$. Then (a suitable restriction of) $\gamma$ coincides with $\gamma_{\mathfrak c}$ up to a translation in time.\end{proposition}

In order to prove Proposition \ref{prop:translation-in-time}, we need the following auxiliary 

\begin{lemma}\label{lemma:canonical_time_2} Let $\gamma \colon I \to \R^2$ be a solution of the ordinary differential equation 
\begin{equation*}
\gamma^\prime(t) = b(\gamma(t)), \qquad t \in I \subset \R,
\end{equation*}
where $I=[0,T]$ and $\frac{1}{\vert b \vert} \in L^{1}_{\rm loc}(\H^{1} \rest \gamma(I))$. Assume that there exists a injective curve $\hat{\gamma}$ defined on $I$ such that $\gamma(I) \subset \hat{\gamma}(I)$ and that $\dot{\hat{\gamma}}=b(\hat{\gamma})$. Then 
\begin{equation*}
\int_{\gamma\left([0,T]\right)} \frac{d\H^{1}(w)}{\vert b(w)\vert} = T - \L^{1} \big( \big\{ t \in [0,T]: \, \gamma^{\prime}(t) = 0 \big\} \big). 
\end{equation*}
\end{lemma}

\begin{proof} Observe that 
\begin{equation*}
\begin{split}
\int_{\gamma\left([0,T] \right)} \frac{d\H^{1}(w)}{\vert b(w)\vert} & \stackrel{(1)}{=} \int_{\gamma\left([0,T]\right)} \frac{\1_{\{b \ne 0\}}(w) \, d\H^{1}(w)}{\vert b(w)\vert} \\ 
& \stackrel{(2)}{=} \int_{\{ t \in [0,T]: \, \gamma^{\prime}(t) \ne 0\}} \frac{\vert \gamma^{\prime}(\tau)\vert}{\vert b (\gamma(\tau))\vert }d\tau \\
& =~ T - \L^{1} \big( \big\{ t \in [0,T]: \, \gamma^{\prime}(t) = 0 \big\} \big),
\end{split}
\end{equation*}
where 
\begin{enumerate}
\item[(1)] follows by definition; 
\item[(2)] is the Area formula, i.e. $\H^{1} \rest C = \gamma_{\#}(\vert \gamma^{\prime} \vert \L^{1})$, where $C=\gamma((0,T))$, which can be applied because there exists $\hat{\gamma}$ by hypothesis.\end{enumerate}
This concludes the proof. \end{proof}

Now we can prove Proposition \ref{prop:translation-in-time}.  

\begin{proof} Let $\overline{t} \in [0,T]$ such that $\gamma_{\mathfrak c}(0)=\gamma(\overline{t})$. By Lemma \ref{lemma:canonical_time_2}, we have that for any $s$ in a suitable subinterval of $[0,T]$ it holds 
\begin{equation}\label{eq:gamma-sono-traslate-in-t}
\int_{\gamma\left([\overline{t}, \overline{t}+s \right])} \frac{d\H^{1}(w)}{\vert b(w)\vert} = (\overline{t}+s) -\overline{t} - \L^{1}([\overline{t},\overline{t}+s] \cap \gamma^{-1}(\{b=0\})).
\end{equation}
By Lemma \ref{lemma:eta_M} and the fact that $\L^2(\{b=0\} \setminus M) = 0$, where $M$ is defined in \eqref{eq:set_M}, we know that for $\eta$-a.e. $\gamma \in \tilde{\Gamma}$, 
\begin{equation*}
\L^{1} \big( \big\{ t \in [0,T]: \, \gamma(t) \in \{b=0\} \big\} \big) = 0, 
\end{equation*}
hence \eqref{eq:gamma-sono-traslate-in-t} is actually 
\begin{equation}\label{eq:canonical-length-gamma}
\int_{\gamma\left([\overline{t}, \overline{t}+s] \right)} \frac{d\H^{1}(w)}{\vert b(w)\vert} = s.
\end{equation}
On the other hand, applying again Lemma \ref{lemma:canonical_time_2} to $\gamma_{\mathfrak c}$, which is injective, we get 
\begin{equation}\label{eq:canonical-length-gamma-a}
\int_{\gamma_{\mathfrak c}(0,s)} \frac{d\H^{1}(w)}{\vert b(w)\vert} = s.
\end{equation}
Since, by definition, $\gamma_{\mathfrak c}(0)=\gamma(\overline{t})$, comparing \eqref{eq:canonical-length-gamma} and \eqref{eq:canonical-length-gamma-a} and using the fact that $\vert b \vert > 0$ $\H^1$-a.e. on $\gamma$, we deduce that 
\begin{equation*}
\gamma(\overline{t} + s) =  \gamma_{\mathfrak c}(s)
\end{equation*}
which means that $\gamma$ (restricted to a suitable time subinterval of $[0,T]$) and $\gamma_{\mathfrak c}$ coincide up to a translation in time. \end{proof}

\subsection{Covering property of the regular level sets}

Let us recall that for each ball $B \in \mathscr B$ and for any rational numbers $s,t \in \Q \cap (0,T)$ with $s<t$ we have set
\begin{equation*}
\mathsf T_{B,s,t} := \left\{ \gamma \in \Gamma_{B}: \,\gamma(s) \notin B, \gamma(t) \notin B\right\}.
\end{equation*}

\begin{remark} In the same way as in Remark \ref{rem:unione-tb-tutto-moving}, we can easily see that
\begin{equation}\label{eq:b-sigma-theta}
\bigcup_{\substack{B \in \B \\s,t \in \Q \cap [0,T]}} \mathsf T_{B,s,t} = \tilde{\Gamma}.
\end{equation}
\end{remark}

For each $B\in \ss B$, $s\in \Q \cap (0,T)$, $t \in \Q \cap (s,T)$ restricting $\eta$ to ${\sf T}_{B,s,t}$, we can construct the local Hamiltonian $H_{B,s,t}$ as in Sections~\ref{Ss_parti_curves}-\ref{ss:disintegration-lebesgue}.

We now set 
\begin{equation}\label{e-E}
\hat E := \bigcup_{\substack{B \in \B \\s,t \in \Q \cap [0,T]}} E_{B,s,t}^\star.
\end{equation}

The following covering property is a global analog of Lemma~\ref{l-E-covers-nonzero-b}:
\begin{lemma} \label{l-global-E-covers-nonzero-b}
It holds that $\hat E \supset \{b \ne 0\} \mod \L^2$.
\end{lemma}

\begin{proof}
Let $P := \{b \ne 0\} \setminus \hat E$. Then for any $B \in \ss B$ it holds that $P \subset \{\nabla H_B = 0\} \mod \L^2$. Since $b \ne 0$ on $P$ and $\nabla H^\perp = r_B b$ it holds that $r_B = 0$ a.e. on $P$ for all $B \in \ss B$. Then for any $B \in \ss B$
\begin{equation*}
\begin{aligned}
0 &= \int_{P \cap B} r_B \, dx \\
  &= \int_0^T\int \1_{P\cap B}(x)\rho_B(t,x) \, dx \, dt \\
  &= \int_{\tilde \Gamma} \int_0^T\1_{P\cap B}(\gamma(t)) \, d\eta(\gamma) \, dt,
\end{aligned}
\end{equation*}
hence $\eta$-a.e. $\gamma \in \tilde \Gamma$ spends zero amount of time in $P\cap B$.
Since $B$ is arbitrary and $\ss B$ is countable, we can generalize this claim to the whole set $P$:
\begin{equation}\label{e-moving-curves-do-not-stay-in-P}
\int_{\tilde \Gamma} \int_0^T \1_P(\gamma(t)) \, dt \, d\eta(\gamma) = 0.
\end{equation}

By nearly incompressibility
\begin{equation*}
\begin{aligned}
\L^2(P) &\stackrel{\phantom{(***)}}{\le} C \int_0^T \int \1_P(x) \rho(t, x) \, dx \, dt \\
        & \stackrel{\phantom{(***)}}{=} C \int_0^T \int_{\dot \Gamma \cup \tilde \Gamma}  \1_P(\gamma(t)) \, d\eta(\gamma) \, dt \\
        & \stackrel{(*)\phantom{**}}{=} C \int_0^T \int_{\dot \Gamma}  \1_P(\gamma(t)) \, d\eta(\gamma) \, dt \\
        & \stackrel{(**)\phantom{*}}{=} C \int_0^T \int_{\dot \Gamma}  \1_P(\gamma(t)) \1_{\{b=0\}}(\gamma(t)) \, d\eta(\gamma) \, dt \\
        & \stackrel{\phantom{(***)}}{\le} C \int_0^T \int \1_P(\gamma(t)) \1_{\{b=0\}}(\gamma(t)) \, d\eta(\gamma) \, dt \\
        & \stackrel{\phantom{(***)}}{\le} C \int_0^T \int \1_P(x) \1_{\{b=0\}}(x) \rho(t,x) \, dx \, dt \\
        & \stackrel{(***)}{=} 0,
\end{aligned}
\end{equation*}
where 
\begin{itemize}
\item (*) holds by \eqref{e-moving-curves-do-not-stay-in-P};
\item (**) holds because $\1_{\{b=0\}}(\gamma(t))=1$ for any $t\in[0,T]$ and any $\gamma \in \dot \Gamma$: indeed, for any $\gamma \in \dot \Gamma$ which is an integral curve of $b$ we have $0=\gamma'(t)=b(\gamma(t))$, hence $\gamma(t)\in \{b=0\}$;
\item (***) holds because $P$ and $\{b=0\}$ are disjoint. \qedhere
\end{itemize}
\end{proof}

In view of Corollary~\ref{c-weak-Sard} the proof above actually leads to a stronger statement:
\begin{lemma}\label{l-E-is-nonzero-b}
$\hat E = \{b \ne 0\} \mod \L^2$.
\end{lemma}




\section{Solution of the transport equation on integral curves}
\label{s:new-eight}

\subsection{Splitting on the level sets of the time-dependent problem}

We now present the time-dependent version of Lemmas \ref{disint-div-ub}-\ref{cc-lemma}-\ref{separation}-\ref{param-lemma}.

\begin{lemma}\label{lemma:splitting-eq-trans-on-cc} Fix a ball $B \in \mathscr B$ and the corresponding Hamiltonian $H_B$. Let $v \in L^\infty([0,T] \times B)$ be a solution to the problem 
\begin{equation}\label{eq:probl-cont-varrho}
\begin{cases}
v_{t} + \dive( v b) = 0, \\
v(0, \cdot) = v_{0}(\cdot),
\end{cases} \qquad \text{in } \ss D^{\prime}((0,T) \times B)
\end{equation}
Then $\hat{v}(t,s) := v(t,\gamma(s))$ solves 
\begin{equation*}
 \begin{cases}
 \partial_{t} \big( \hat{v} \hat{c}_{h} \vert \hat b  \vert \big) + \partial_{s} \big( \hat{v}\hat{c}_{h}\vert \hat{b} \vert \big) = 0, \\ 
\hat{v}(0,\cdot) = \hat{v}_{0}(\cdot ),
\end{cases}\qquad \text{in } \ss D^{\prime}((0,T) \times I).
\end{equation*}
for $\L^1$-a.e. $h$, where $\gamma \colon I \to \R^2$ is an admissible parametrization of a connected component of the level set $E_h$ of the Hamiltonian $H_B$. 
\end{lemma}

\begin{proof}
Multiplying by a function $\psi \in C_{c}^{\infty}([0,T))$ and formally integrating by parts we get
\begin{equation*}
 v_{t} \psi + \dive (v \psi b) = \psi \nu \, \Rightarrow \, \dive \left( \int_{0}^{T} v \psi \, dt\, b \right) =\int_{0}^{T} v \psi_{t}\, dt - \psi(0){v}_{0}, 
\end{equation*}
i.e. 
$$
\dive (wb) = \mu,
$$
where $w:=\int_{0}^{T}v \psi \, dt$ and 
\begin{equation*}
\mu:=\left( \int_{0}^{T} v \psi_{t}\, dt - \psi(0){v}_{0}\right)\L^{2}. 
\end{equation*}


Applying Lemma \ref{disint-div-ub}, we obtain that continuity equation implies
\begin{equation}\label{eq:cont-eq-on-cc}
\dive \big( w c_{h} b \H^1 \rest E_h \big) = \mu_{h} \quad \text{ in } \ss D^{\prime}(\R^2) \text{ for } \L^1 \text{-a.e. } h \in \R.
\end{equation}

The measure $\mu_{h}$ can be computed explicitly, using Coarea Formula: 
\begin{equation*}
\mu_{h} = \bigg( \int_{0}^{T} v \psi_{t}\, dt -  \psi(0)v_{0} \bigg) \H^{1} \rest E_h. 
\end{equation*}

Thanks to Lemma \ref{param-lemma}, equation \eqref{eq:cont-eq-on-cc} is \emph{equivalent} to
 \begin{equation*}
\partial_{s} \big( \hat{v}\hat{c}_{h}\vert \hat{b} \vert \big) = \hat{\mu}_{h},
\end{equation*}
 in $\ss D^{\prime}((0,T) \times I)$. Now being $\gamma_{h}$ Lipschitz and injective, we have
\begin{equation*}
(\gamma_{h}^{-1})_{\#} \big( \H^{1} \rest E_{h} \big) = \vert \gamma_{h}^{\prime} \vert \L^{1},
\end{equation*}
and this allows us to compute explicitly
\begin{equation*}
\begin{split}
\hat{\mu}_{h} &= (\gamma_{h}^{-1})_{\#} \mu_{h} \\
 &= (\gamma_{h}^{-1})_{\#} \left( \int_{0}^{T} v \psi_{t}\, dt \, c_{h} \H^{1} \rest E_{h} - \int_{\R^2}\psi(0) v_{0} c_{h} d\H^{1} \rest F_{h} \right)\\
&= \int_{0}^{T} v(\tau, \gamma(s)) \psi_{\tau}(\tau) c_{h}(\gamma_{h}(s))\vert b(\gamma_{h}(s))\vert \, d\tau - \psi(0)v_{0}(\gamma_{h}(s))c_{h}(\gamma(s)),
 \end{split}
 \end{equation*}
which formally means
\begin{equation*}
 \hat{\mu}_{h} = - \int_0^T \partial_{t} \big( \hat{v} \vert \hat b \vert \hat{c}_{h} \big).
\end{equation*}
To sum up, we have obtained that Problem \eqref{eq:probl-transp} implies that
\begin{equation*}
 \begin{cases}
 \partial_{t} \big( \hat{v} \hat{c}_{h} \vert \hat b  \vert \big) + \partial_{s} \big( \hat{v}\hat{c}_{h}\vert \hat{b} \vert \big) = 0, \\ 
\hat{v}(0,\cdot) = \hat{v}_{0}(\cdot ), 
\end{cases}
\end{equation*}
in $\ss D^{\prime}((0,T) \times I)$ for $\L^1$-a.e. $h \in \R$. \end{proof}

\begin{lemma}\label{l-local-transport}
Fix $\sigma \in \Q \cap (0,T)$, $\theta \in \Q \cap (\sigma, T)$ and $B \in \ss B$. Let $H:=H_{B,\sigma,\theta}$.
Let $u \in L^\infty([0,T] \times \R^2)$ be a $\rho$-weak solution of the problem 
\begin{equation*}
\begin{cases}
u_{t} + b \cdot \nabla u = 0, \\
u(0, \cdot) = u_{0}(\cdot),
\end{cases} \qquad \text{in } \ss D^{\prime}((0,T) \times \R^2).
\end{equation*}
Then there exists a negligible set $Z=Z_{B,\sigma,\theta} \subset \R$ such that 
\begin{itemize}
\item 
for any $h \in Z$ the level set $E_h:=H^{-1}(h)$ is regular;
\item 
if $h\notin Z$ and $E_h$ is regular then
for any nontrivial connected component $C$ of $E_h$ with admissible parametrization $\gamma_C \colon I \to \R^2$,
any $t\in(0,T)$ and any $s\in I$
there exists a constant $w$ such that
\begin{equation}
u(t+\xi, \gamma_C(s+\xi)) = w
\end{equation}
for a.e. $\xi\in \R$ such that $s+\xi \in I$ and $t + \xi \in (0,T)$.
\newline
In particular, for any $s \in I$ it holds that
\begin{equation}
u(\xi, \gamma_C(s+\xi)) = u_0(s)
\end{equation}
for a.e. $\xi \in \R$ such that $s+\xi \in I$.
\end{itemize}
\end{lemma}

\begin{proof}
Setting $v:=u\rho \in L^\infty([0,T] \times \R^2)$ and $v_{0}(\cdot) = u_0(\cdot)\rho(0,\cdot)$, by definition of $\rho$-weak solution
we have
\begin{equation}
\begin{cases}
v_{t} + \dive( v b) = 0, \\
v(0, \cdot) = v_{0}(\cdot),
\end{cases} \qquad \text{in } \ss D^{\prime}((0,T) \times \R^2).
\end{equation}
Hence we can apply Lemma \ref{lemma:splitting-eq-trans-on-cc} in $B$ to get 
\begin{equation}\label{eq:prob-1-d}
 \begin{cases}
 \partial_{t} \big( \hat{v} \hat{c}_{h} \vert \hat b  \vert \big) + \partial_{s} \big( \hat{v}\hat{c}_{h}\vert \hat{b} \vert \big) = 0, \\ 
\hat{v}(0,\cdot) = \hat{v}_{0}(\cdot ),
\end{cases}\qquad \text{in } \ss D^{\prime}((0,T) \times I).
\end{equation}
for all $h \in H(B) \setminus N_1$, where $\L^1(N_1)=0$.

From \eqref{eq:prob-1-d} it immediately follows that the function
\begin{equation}\label{e-wrk1}
\xi \mapsto \Big( \hat{\rho} \hat{u} \hat{c}_h \vert \hat b \vert \Big)(t+\xi, s+\xi)
\end{equation}
is equal a.e. to some constant $w_1$.

Applying the same argument to the problem 
\begin{equation}\label{e-wrk2}
\begin{cases}
\rho_{t} + \dive( \rho b) = 0, \\
\rho(0, \cdot) = \rho_{0}(\cdot),
\end{cases} \qquad \text{in } \ss D^{\prime}((0,T) \times \R^2),
\end{equation}
(which holds thanks to nearly incompressibility assumption) we obtain
a negligible set $N_2$
such that 
for all $h \in H(B) \setminus N_2$, for any connected component of $E_h$
the map 
\begin{equation}\label{e-wrk1}
\xi \mapsto \Big( \hat{\rho} \hat{c}_h \vert \hat b \vert \Big)(t+\xi, s+\xi)
\end{equation}
is equal a.e. to some constant $w_2$.

Let $N:=N_1 \cup N_2$ and fix $h \notin N$.

Comparing \eqref{e-wrk1} and \eqref{e-wrk2}, using that $\rho c_h \vert b \vert>0$ $\H^1$-a.e. on $E_h$ (for a.e. $h$), we obtain that 
\begin{equation}\label{e-wrk1}
\xi \mapsto \Big( \hat{v} \hat{c}_h \vert \hat b \vert \Big)(t+\xi, s+\xi)
\end{equation}
is equal a.e. to the constant $w = w_1/w_2$ for a.e. $h \notin N$.
\end{proof}

\subsection{Selection of appropriate trajectories}
\begin{lemma}\label{l-covering-with-cc}
There exists an $\eta$-negligible set $N \subset \Gamma$ such that any integral curve $\gamma\in \tilde \Gamma \setminus N$
of the vector field $b$ has the following properties:
\begin{enumerate}
\item for any $B \in \ss B$, if $\gamma \in \mathsf T_{B,s,t}$ then each connected component of $\gamma([s,t])\cap B$ is contained in a regular level set of $H_B$;
 
\item for any $\tau \in (0,T)$ there exist a ball $B \in \ss B$, $s\in \Q \cap (0,T)$ and $t \in \Q \cap (\tau,T)$ such that $\gamma \in \mathsf T_{B,s,t}$.
\end{enumerate}
\end{lemma}

\begin{proof}
First of all, using Lemma \ref{lemma:eta_M} we can remove a negligible set of integral curves of $b$ which stay in the set $\{b=0\}$ for a positive amount of time. 
Applying Lemmas \ref{lemma:library} and \ref{lemma:library-plus} countably many times (for each ball $B \in \ss B$ and all rationals $s \in \Q \cap (0,T)$ and $t \in \Q \cap (s,T)$)
we obtain the set $N\subset \Gamma$ such that the first property holds.

Next, for any $\tau \in (0,T)$ there exists $s \in \Q \cap (0,\tau)$ such that $\gamma(s) \ne \gamma(\tau)$.
(Otherwise, since $\gamma$ is an integral curve of $b$, it would have to stay in $\{b=0\}$ for a positive amount of time).
Similarly there exists $t \in (s, T)$ such that $\gamma(t) \ne \gamma(\tau)$. Then for any ball $B \in \ss B$ with sufficiently small radius,
containing $\gamma(\tau)$ and not containing $\gamma(s)$ and $\gamma(t)$ it clearly holds that $\gamma \in T_{B,s,t}$.
\end{proof}

\begin{lemma} \label{lemma:negligible}
Let $Z_{B,s,t}$ denote negligible set given by Lemma~\ref{l-local-transport}.
Then for $\eta$-a.e. $\gamma \in \tilde \Gamma$ it holds that
\begin{equation}\label{e-agree-with-transport}
H_{B,s,t}(\gamma([0,T])) \cap Z_{B,s,t} = \emptyset.
\end{equation}
\end{lemma}

\begin{proof}
Set $A:=H_{B,s,t}^{-1}(Z_{B,s,t})$: by Coarea Formula, $\L^2(A) = 0$. Applying Lemma \ref{lemma:eta_lebesgue} we deduce that 
\begin{equation*}
\eta \Big(\left\{\gamma \in \Gamma: \L^{1}(\{ t \in [0,T]: \gamma(t) \in A\})>0\right\} \Big)= 0.
\end{equation*}
On the other hand $b \ne 0$ a.e. on $E_{B,s,t}$, hence
\begin{equation}
\begin{split}
& \Big\{ \gamma \in \tilde{\Gamma} \setminus N: \, \gamma([0,T]) \cap E_{B,s,t} \subset E_h, \, h \in Z \Big\} \\ 
& = \Big\{ \gamma \in \tilde{\Gamma} \setminus N: \, \gamma([0,T]) \cap E_{B,s,t} \subset A \Big\} \\
& \subset \left\{\gamma \in \Gamma: \L^{1}(\{ t \in [0,T]: \gamma(t) \in A\})>0\right\}.\qedhere
\end{split}\end{equation}\end{proof}


From the Lemma \ref{lemma:negligible} it does not follow immediately that the endpoints $\gamma(0)$ and $\gamma(T)$ are
contained in regular level sets of some Hamiltonians. But now we are going to establish this property.
Being $Z_{B,s,t}$ given by Lemma~\ref{l-local-transport}, let $\tilde E_{B,s,t} := E_{B,s,t} \setminus H_{B,s,t}^{-1}(Z_{B,s,t})$
and
\begin{equation}
\tilde E := \bigcup_{\substack{B \in \B, \\ s,t \in \Q \cap (0,T): \; s<t}} \tilde E_{B,s,t}.
\end{equation}
Note that since $\tilde E_{B,s,t}= E_{B,s,t} \mod \L^2$ (by Coarea formula), it follows that $\tilde E = \hat E \mod \L^2$.

The following lemma shows that $\eta$-a.e. nontrivial trajectory of $b$ starts from the set $\tilde E$ (and also stops in $\tilde E$):
\begin{lemma}\label{l-eta-ae-gamma-starts-from-E}
For $\eta$-a.e. $\gamma \in \tilde \Gamma$ it holds that $\gamma(0)\in \tilde E$ and $\gamma(T) \in \tilde E$.
\end{lemma}
\begin{proof}
Consider the set $X$ of $\eta \in \tilde \Gamma$ such that $\gamma(0)\not\in \tilde E$.
By Lemma~\ref{l-E-is-nonzero-b} it holds that $b=0$ a.e. on the complement of $\tilde E$.
Hence by Lemma~\ref{lemma:eta_M} we have $\eta(X)=0$. The argument for $\gamma(T)$ is similar.
\end{proof}

In the lemmas above we have been removing $\eta$-negligible sets of trajectories of $b$. Let us summarize some properties of
the remaining ones:
\begin{lemma}\label{l-ae-gamma-in-E}
There exists a $\eta$-negligible set $R\subset \tilde \Gamma$ such that for any $\tau \in [0,T]$ and any $\gamma \in \tilde \Gamma \setminus R$ there exist $s \in \Q\cap (0,T)$, $t \in \Q \cap (s,T)$ and $B\in \ss B$ such that $\gamma(\tau) \in \tilde E_{B,s,t}$.
\end{lemma}
\begin{proof}
We define $R$ as the union of $\eta$-negligible sets given by Lemmas \ref{l-covering-with-cc}, \ref{lemma:negligible} and \ref{l-eta-ae-gamma-starts-from-E}.
If $\tau \in (0,T)$ the claim follows from Lemma \ref{l-covering-with-cc} since we can always find $s$ and $t$ such that $\tau \in (s,t)$
and the desired property holds. If $\tau = 0$ or $\tau = T$ then the result follows from Lemma~\ref{l-eta-ae-gamma-starts-from-E}.
\end{proof}

\begin{corollary}\label{c-local-conservation}
For any $\gamma \in \tilde \Gamma \setminus R$ and any $\tau \in [0,T]$ there exists $\delta > 0$ and a constant $w$
such that the function $\xi \mapsto u(\xi, \gamma(\xi))$ is equal to $w$ for a.e. $\xi \in (\tau-\delta, \tau+\delta) \cap [0,T]$.
Moreover, if $\tau = 0$ then the constant $w$ is equal to $u_0(\gamma(0))$.
\end{corollary}

\begin{proof}
The result follows directly from Lemma~\ref{l-ae-gamma-in-E}, Proposition~\ref{prop:translation-in-time} and Lemma~\ref{l-local-transport}.
\end{proof}

\subsection{Solutions are constant along $\eta$-a.e. trajectory}

Now we are in a position to recover the method of characteristics in our weak setting:

\begin{lemma}\label{l-characteristics-1} Suppose that $b$ is a bounded, autonomous, $\BV$ compactly supported, nearly incompressible (with density $\rho$) vector field on $\R^2$ and let $u \in L^\infty([0,T] \times \R^2)$ be a $\rho$-weak solution of the problem 
\begin{equation}\label{eq:probl-transp}
\begin{cases}
u_{t} + b \cdot \nabla u = 0, \\
u(0, \cdot) = u_{0}(\cdot),
\end{cases} \qquad \text{in } \ss D^{\prime}((0,T) \times \R^2).
\end{equation}
Then for $\eta$-a.e. $\gamma \in \Gamma$ for a.e. $t \in [0,T]$ it holds that
\begin{equation*}
u(t,\gamma(t)) = u_0(\gamma(0)).
\end{equation*}
\end{lemma}

\begin{proof}
It is clear that the thesis holds for any $\gamma \in \dot \Gamma$. Indeed, by Proposition~\ref{prop:locality}
\begin{equation*}
\partial_t (\rho u \1_M ) = 0
\end{equation*}
in $\ss D'$, where the set $M$ is defined in \eqref{eq:set_M}.

Hence it is sufficient to consider only the moving trajectories, i.e. $\gamma \in \tilde \Gamma$.
Let $R$ be the set given by Lemma \ref{l-ae-gamma-in-E}. Let $\gamma \in \tilde \Gamma \setminus R$.
By Corollary~\ref{c-local-conservation} for any $\tau \in [0,T]$ there exists $\delta>0$
such that the function $t \mapsto u(t,\gamma(t))$ is equal to some constant $w_\tau$ for a.e. $t \in (\tau - \delta, \tau + \delta) \cap[0,T]$.
Moreover, if $\tau = 0$ then $w_\tau = u_0(\gamma(0))$. It remains to extract a finite covering of $[0,T]$.
%
%
%
%
 \end{proof} 

The following lemma is elementary, we prove it for sake of completeness. 

\begin{lemma} \label{l-characteristics-2} Let $u \in L^\infty([0,T] \times \R^2)$. 
If for $\eta$-a.e. $\gamma$ and a.e. $t\in[0,T]$ it holds that $u(t,\gamma(t)) = u_0(\gamma(t))$, then $u$
solves the transport equation with the initial condition $u_0$, i.e. 
\begin{equation*}
\begin{cases}
u_{t} + b \cdot \nabla u =0, \\
u(0, \cdot) = u_{0}(\cdot).
\end{cases}
\end{equation*}
\end{lemma}

\begin{proof}
Let $\varphi \in C_c^\infty([0,T) \times \R^2)$ be a smooth test function which vanishes at $T$. Then
\begin{equation*}
\begin{split}
& \int_{0}^T \int_{\R^2} (\rho u \varphi_t + \rho u b \nabla \varphi) \, dx \, dt + \int_{\R^2} \rho(0, x) u_0(x) \varphi(0, x) \, dx \\
& = \int_0^T \int_{\Gamma} u(t,\gamma(t)) \partial_t \varphi(t, \gamma(t)) \, d\eta(\gamma) \, dt + \int_{\Gamma} u_0(\gamma(0)) \varphi(0, \gamma(0)) \, d\eta(\gamma) \\
& = \int_0^T\int_{\Gamma} u_0(\gamma(0)) \partial_t \varphi(t, \gamma(t)) \, d\eta(\gamma) \, dt + \int_{\Gamma} u_0(\gamma(0)) \varphi(0, \gamma(0)) \, d\eta(\gamma) \\
&= - \int_{\Gamma} u_0(\gamma(0)) \varphi(0, \gamma(0)) \, d\eta(\gamma) + \int_{\Gamma} u_0(\gamma(0)) \varphi(0, \gamma(0)) \, d\eta(\gamma)
=0.\qedhere
\end{split}
\end{equation*}
\end{proof}

\section{Renormalization: proof of the Main Theorem}
\label{s:new-nine}

We are finally ready to state and prove the main result of the paper, which is the following 
\begin{theorem}
\label{T_main}
Every bounded, autonomous, compactly supported and nearly incompressible $\BV$ vector field on $\R^2$ has the renormalization property.
\end{theorem}

\begin{proof}
Let $u \in L^\infty([0,T] \times \R^2)$ be a solution of 
\begin{equation*}
\begin{cases}
u_{t} + b \cdot \nabla u = 0, \\
u(0, \cdot) = u_{0}(\cdot),
\end{cases} \qquad \text{in } \ss D^{\prime}((0,T) \times \R^2).
\end{equation*}

By Lemma~\ref{l-characteristics-1} the function $t \mapsto u(t,\gamma(t))$ is constant for $\eta$-a.e. $\gamma$. Then for any $\beta \in C^1(\R, \R)$ the function $t \mapsto \beta(u(t,\gamma(t)))$ is constant for $\eta$-a.e. $\gamma$. Hence by Lemma~\ref{l-characteristics-2} the function $\beta(u)$ is a solution of 
\begin{equation*}
\begin{cases}
(\beta(u))_{t} + b \cdot \nabla \beta(u)=0, \\
\beta(u) (0, \cdot) = \beta(u_{0})(\cdot).
\end{cases}
\end{equation*}
This concludes the proof.
\end{proof}

\bibliographystyle{plain}
\bibliography{biblio}

\end{document}